\documentclass[11pt]{article}
\usepackage[margin=1in]{geometry}
\usepackage{fontenc,fontaxes}
\usepackage{amsthm}
\usepackage{amssymb}
\usepackage{lmodern}
\usepackage{authblk}
\usepackage[authoryear,round]{natbib}
\usepackage{mathtools}
\usepackage{tikz}
\usepackage[bf]{caption}
\usepackage{bbm}
\usepackage{setspace}
\usepackage[greek,english]{babel}

\usepackage[bookmarksnumbered,bookmarksopen,colorlinks,linkcolor={black},urlcolor={black},citecolor={black}]{hyperref}

\newtheorem{lemma}{Lemma}[section]
\newtheorem{proposition}{Proposition}[section]

\newtheorem*{AFT}{Analytic Fredholm Theorem}

\theoremstyle{definition}
\newtheorem{remark}{Remark}[section]
\newtheorem{assumption}{Assumption}[section]

\mathtoolsset{mathic=true}
\DeclareMathOperator{\ran}{ran}

\DeclareMathOperator{\id}{id}

\DeclareMathOperator{\FF}{\mathcal{F}}

\newcommand{\NZ}[1]{N_{#1}}
\newcommand{\NO}[1]{N_{#1}}
\newcommand{\AZ}[1]{A_{#1}}

\newcommand{\PZ}[1]{\mathrm{P}_{\mathrm{{R}}^c_0}}
\newcommand{\PZZ}[1]{P_{(\ran \AZ{0}+\AZ{1}\ker\AZ{0})^c}}

\newcommand{\NC}[1]{N_{#1}}

\newcommand{\SZ}[1]{S_{\mathrm{{R}}^c_0}}
\newcommand{\SZZ}[1]{S_{\mathrm{{R}}^c_1}}

\newcommand{\RR}{\mathrm{{R}}_0}
\newcommand{\RRR}{\mathrm{{R}}_1}
\newcommand{\KK}{\mathrm{{K}}_0}
\newcommand{\KKK}{\mathrm{{K}}_1}
\newcommand{\RC}{\mathrm{{R}}^c_0}
\newcommand{\RCC}{\mathrm{{R}}^c_1}
\newcommand{\KC}{\mathrm{{K}}^c_0}
\newcommand{\KCC}{\mathrm{{K}}^c_1}
\newcommand{\PR}{\mathrm{P}_{\RC}}
\newcommand{\PK}{\mathrm{P}_{\KC}}
\newcommand{\PRR}{\mathrm{P}_{\RCC}}
\newcommand{\PKK}{\mathrm{P}_{\KCC}}
\newcommand{\PP}{\mathrm{P}}
\newcommand{\QL}{Q^L_{\{\RC,\KC\}}}
\newcommand{\QR}{Q^R_{\{\RC,\KC\}}}
\newcommand{\SG}{(\SZ{r_1})^g_{\{\RCC,\KCC\}}}
\newcommand{\SSG}{(S_{V_0})^g_{\{V_1,W_1\}}}
\newcommand{\indj}{\mathbbm{1}_{j=0}}
\newcommand{\GZ}{G_{j}}
\numberwithin{equation}{section}

\expandafter
\def \expandafter \normalsize \expandafter{\normalsize \setlength \abovedisplayskip{5pt plus 2pt minus 3pt}}
\expandafter
\def \expandafter \normalsize \expandafter{\normalsize \setlength \abovedisplayshortskip{0pt plus 2pt}}
\expandafter
\def \expandafter \normalsize \expandafter{\normalsize \setlength \belowdisplayskip{5pt plus 2pt minus 3pt}}
\expandafter
\def \expandafter \normalsize \expandafter{\normalsize \setlength \belowdisplayshortskip{2pt plus 2pt}}

\begin{document}
	
	\newcommand{\gpi}{\textrm{\greektext p}}	
	
	\title{Fredholm inversion around a singularity: application to autoregressive time series in Banach space}
	\date{}
	\author{Won-Ki Seo\footnote{This paper is based on a chapter of the author's Ph.D.\ dissertation at the University of California, San Diego. Author's address : University of Sydney, NSW 2006, Australia. Email: {won-ki.seo@sydney.edu.au}} }
	\affil{University of Sydney}
	
	\maketitle
	
	\begin{abstract}	
		This paper consider inverting a holomorphic Fredholm operator pencil. Specifically, we provide necessary and sufficient conditions for the inverse of a holomorphic Fredholm operator pencil to have a simple pole and a second order pole. Based on those results, a closed-form expression of the Laurent expansion of the inverse around an isolated singularity is obtained in each case. As an application of the results, we obtain a suitable extension of the Granger-Johansen representation theory for random sequences taking values in a separable Banach space. Due to our closed-form expression of the inverse, we may fully characterize the solutions to a given autoregressive law of motion except a term that depends on initial values.
	\end{abstract}
  \vspace{1em} 		 \vspace{1em}

	\pagebreak
	\section{Introduction}
	The so-called Granger-Johansen representation theory is the results on the existence and representation of the solutions to a given autoregressive law of motion. Due to contributions by \cite{engle1987}, \cite{Johansen1991,Johansen1996}, \cite{Schumacher1991}, and \cite{Faliva2010}, we already have well developed representation theory in finite dimensional Euclidean space. It worths mentioning that, in the latter two papers, the Granger-Johansen theory is obtained in the framework of analytic function theory;  \cite{Schumacher1991} obtains a necessary and sufficient condition for a matrix-valued function of a single complex variable (matrix pencil), which characterizes an autoregressive law of motion, to have a simple pole at one, and shows that this leads to nonstationary I(1) solutions, which become stationary by first-order differencing.  The monograph of \cite{Faliva2010} provides a systematic reworking and extension of \cite{Schumacher1991}, and contains a representation theorem associated with nonstationary I(2) solutions, which become stationary by second-order differencing; see also  \cite{franchi2016} and  \cite{Franchi2017a} for more general results on this topic. 
	More recently, the Granger-Johansen representation theory was extended to infinite dimensional function spaces (see e.g., see \citealp{BSS2017}) and, as in  \cite{Schumacher1991}, it turns out that the representation theory can be obtained by inverting the operator pencil which characterizes the autoregressive law of motion at an isolated singularity; see e.g., \cite{Seo2017}, \cite{BS2018}, \cite{Franchi2017b} and  \cite{seo_2022}. Of course, this is certainly not the only example where inversion of operator pencil can be useful in applied fields. 
	
	
	
	In this paper, we consider inverting holomorphic Fredholm operator pencils around an isolated singularity. Specifically, we first obtain necessary and sufficient conditions for the inverse of a holomorphic Fredholm pencil to have a simple pole and a second order pole. We then obtain a closed-form expression of the inverse by deriving  a recursive formula to determine all the coefficients in the Laurent expansion of the inverse around an isolated singularity. We apply our theoretical results to obtain a suitable version of the Granger-Johansen representation theorem. Due to our closed-form expression, I(1) and I(2) solutions to a given autoregressive law of motion can be fully characterized up to a component depending on initial values even in a separable Banach space setting, which has not been obtained in the literature to the best of the author's knowledge.

	The remainder of the paper is organized as follows. In Section \ref{sprelim}, we review some essential mathematics. In Section \ref{sholomorphic}, we study in detail on inversion of a holomorphic Fredholm pencil based on the analytic Fredholm theorem; our main results are obtained in this section.
	Section \ref{srep} contains a suitable extension of the Granger-Johansen representation theory as an application of our inversion theorems. Conclusion follows in Section \ref{sconclude}.  	
	
	\section{Preliminaries}\label{sprelim}
	\subsection{Review of Banach spaces}
	Let $\mathcal B$ be a separable Banach space over the complex plane $\mathbb{C}$ with norm \(\|\cdot\|\). Moreover let \(\mathcal L_{\mathcal B}\) denote the Banach space of bounded linear operators on $\mathcal B$ with the usual operator norm \(\Vert A\Vert_{\mathcal L_{\mathcal B}}=\sup_{\Vert x\Vert\leq1}\Vert A x\Vert\). Let \(\id_{\mathcal B}\in\mathcal L_{\mathcal B}\) denote the identity map on \(\mathcal B\). Given a subspace \(V\subset \mathcal B\), let \(A{\mid_{V}}\) denote the restriction of an operator \(A\in\mathcal L_{\mathcal B}\) to \(V\). 
	Given $A \in \mathcal L_{\mathcal B}$, we define two important subspaces of $\mathcal B$ as follows:
	\begin{align*}
		\ker A&=\{x\in \mathcal B \,\mid\, A x=0\},\\
		\ran A&=\{A x \,\mid\, x\in \mathcal B\}.
	\end{align*}	
	Let \(V_1,V_2,\ldots,V_k\) be subspaces of \(\mathcal B\). The algebraic sum of $V_1,V_2,\ldots,V_k$ is defined by
	\begin{align*}
		\sum_{j=1}^k V_j = \{v_1+v_2+\ldots,v_k : v_j \in V_j \text{ for each } j  \}.
	\end{align*}  
	We say that \(\mathcal B\) is the (internal) direct sum of \(V_1,V_2,\ldots,V_k\), and write \(\mathcal B = \oplus_{j=1}^k V_j\), if \(V_1,V_2,\ldots,V_k\) are closed subspaces satisfying \(V_j \cap \sum_{j' \neq j} V_{j'} = \{0\}\) and \(\sum_{j=1}^k V_j = \mathcal B\). For any \(V\subset \mathcal B\), we let $V^c \subset \mathcal B$ denote a subspace (if exists) such that \(\mathcal B = V \oplus V^c\). Such a subspace $V^c$ is called a complementary subspace of \(V\). It turns out that a subspace \(V\) allows complementary subspace \(V^c\) if and only if there exists the unique bounded projection onto \(V^c\) along \(V\) \citep[Theorem 3.2.11]{megginson1998}. In general, a complementary subspace is not uniquely determined.

	Given $V \subset \mathcal B$. The cosets of $V$ are the collection of the following sets
	\begin{align*}
		x + V = \{x + v : v \in V \}, \quad x \in \mathcal B.
	\end{align*}
	The quotient space $\mathcal B/V$ is the vector space whose elements are  equivalence classes of the cosets of $V$, with the equivalence relation $\simeq$ is given by
	\begin{align*}
		x + V \simeq y + V \quad \Leftrightarrow \quad x-y \in V.
	\end{align*}
	When $V = \ran A$ for some $A \in \mathcal L_{\mathcal B}$, the dimension of $\mathcal B/V$ is called the defect of $A$. 
	
	\subsection{Fredholm operators}
	An operator $A \in \mathcal L_{\mathcal B}$ is said to be a Fredholm operator if $\ker A$ and $\mathcal B/\ran A$ are finite dimensional.
	The index of a Fredholm operator $A$ is the integer given by $\dim(\ker A)-\dim(\mathcal B/\ran A)$.  It turns out that a bounded linear operator with finite defect has a closed range \citep[Lemma 4.38]{abramovich2002}. Therefore $\ran A$ is closed if \(A\) is a Fredholm operator. Fredholm operators are invariant under compact perturbation; if $A$ is a Fredholm operator and \(K\) is a compact operator, $A+K$ is a Fredholm of the same index.	In this paper, we mainly consider Fredholm operators of index zero, so we let $\FF_0 (\subset \mathcal L_{\mathcal B})$ denote the collection of such operators.
	
	\subsection{Generalized inverse operators}\label{sginverse}
	Let \(\mathcal B_1\) and \(\mathcal B_2\) be Banach spaces and \(\mathcal L_{\mathcal B_1,\mathcal B_2}\) denote the space of bounded linear operators from \(\mathcal B_1\) to \(\mathcal B_2\).  
	In the subsequent discussion, we need the notion of a generalized inverse operator of $A \in \mathcal L_{\mathcal B_1,\mathcal B_2}$. Suppose that \(\mathcal B_1 = \ker A \oplus (\ker A)^c\) and \(\mathcal B_2 = \ran A \oplus (\ran A)^c\).
	Given the direct sum conditions, the generalized inverse of $A$, denoted by $A^g$, is defined as the unique linear extension of $(A{\mid_{ (\ker A)^c}})^{-1}$(defined on \(\ran A\)) to $\mathcal B$.  Specifically, \(A^g\) is given by
	\begin{align} \label{ginverse}
		A^g = (A{\mid_{ (\ker A)^c}})^{-1} (\id_{\mathcal B}-\PP_{(\ran A)^c}),
	\end{align}
	where \(\PP_{V^c}\) denotes the bounded projection onto \(V^c\) along \(V\).
	It can be shown that the generalized inverse \(A^g\) has the following properties:
	\begin{align*}
		&AA^gA = A, \quad A^gAA^g = A^g, \quad A A^g = (\id_{\mathcal B}-\PP_{(\ran A)^c}), \quad A^g A = \PP_{(\ker A)^c}.
	\end{align*}
	Since complementary subspaces are not uniquely determined, \(A^g\) depends on our choice of them. 
	
	\subsection{Operator pencils}
	Let $U$ be an open connected subset of $\mathbb{C}$. A map $A : U \rightarrow \mathcal L_{\mathcal B}$ is called an operator pencil. An operator pencil $A$ is holomorphic at $z_0 \in U$ if the limit 
	\begin{align*}
		A^{(1)}(z_0) := \lim_{z \rightarrow z_0} \frac{A(z) - A(z_0)}{z-z_0}
	\end{align*}  
	exists in the uniform operator topology. If $A$ is holomorphic for all $z \in D \subset U$ for an open connected set $D$, then we say that $A$ is holomorphic on $D$. A holomorphic operator pencil $A$ on $D$ allows the Taylor series for all $z_0 \in D$.

	An operator pencil $A$ is said to be meromorphic on $U$  if there exists a discrete set $U_0 \subset U$ such that $A : U \setminus U_0 \rightarrow \mathcal L_{\mathcal B}$ is holomorphic and the following Laurent expansion is allowed in a punctured neighborhood of $z_0 \in U_0$:
	\begin{align*}
		A(z) = \sum_{j=-m}^{-1} \AZ{j} (z-z_0)^j + \sum_{j=0}^\infty \AZ{j} (z-z_0)^j,
	\end{align*}
	where the first term is called the principal part, and the second term is called the holomorphic part of the Laurent series.  A finite positive integer $m$ is called the order of pole at \(z_0\). 	When \(m=1\) (resp.\ $m=2$), we simply say that \(A(z)\) has a simple pole (resp.\ second order pole) at \(z_0\).  
	If $\AZ{-m}, \ldots, \AZ{-1}$ are finite rank operators, we say that $A(z)$ is finitely meromorphic at $z_0$. In addition, $A(z)$ is said to be finitely meromorphic on $U$ if it is finitely meromorphic at each of its poles. 	
	
	The set of complex numbers \(z\in U\) at which the operator \(A(z)\) is noninvertible is called the spectrum of \(A\), and denoted by \(\sigma(A)\). It turns out that the spectrum is always a closed set \citep[p.\ 56]{Markus2012}.
	
	If \(A(z)\) is a Fredholm operator of index zero for \(z \in U\), we hereafter simply call it an \(\FF_0\)-pencil. 
	
	\subsection{Fredholm Theorem}\label{sfredholm}
	We provides a crucial input, called the analytic Fredholm theorem, for the subsequent discussion. 
	\begin{AFT}{(Corollary 8.4 in \citealp{Gohberg2013})} Let $A : U \to \mathcal L_{\mathcal B}$ be a holomorphic Fredholm operator pencil, and assume that \(A(z)\) is invertible for some element \(z \in U\). Then 
		\begin{itemize}
			\item[$\mathrm{(i)}$] $\sigma(A)$ is a discrete set. 
			\item[$\mathrm{(ii)}$] In a punctured neighborhood of \(z_0 \in \sigma(A)\), 
			\[A(z)^{-1} = \sum_{j=-m}^{\infty}\AZ{j}(z-z_0)^j, \]
			where \(A_0\) is a Fredholm operator of index zero and \(\AZ{-m},\ldots,\AZ{-1}\) are finite rank operators.
		\end{itemize}
	\end{AFT}
	\noindent That is, the analytic Fredholm theorem implies that if the inverse of a holomorphic Fredholm pencil exists, it is finitely meromorphic.

	\subsection{Random elements of Banach space}
	We briefly introduce Banach-valued random variables, called \(\mathcal B\)-random variables. More detailed discussion on this subject can be found in e.g.,  \citet[Chapter 1]{Bosq2000}. We let \(\mathcal B'\) denote the topological dual of \(\mathcal B\). 
	
	Let $(\Omega, \mathbb{F}, \mathbb{P})$ be an underlying probability triple. A $\mathcal B$-random variable is defined as a measurable map $X: \Omega \to \mathcal B$, where $\mathcal B$ is understood to be equipped with its Borel $\sigma-$field. $X$ is said to be integrable if $E\|X\| < \infty$. If $X$ is integrable, there exists a unique element $EX \in \mathcal B$ such that for all $f \in \mathcal B'$, 
	\begin{align*}
		E[f(X)] = f (EX). 
	\end{align*}
	Let $L^2_\mathcal{B}$ denote the space of $\mathcal B$-random variables $X$ such that $EX = 0$ and $E\|X\|^2 < \infty$. 

\subsection{I(1) and I(2) sequences in Banach space}
Let  $\varepsilon = (\varepsilon_t, t\in \mathbb{Z})$ be an independent and identically distributed  sequence in $L^2_{\mathcal B}$ such that $E\varepsilon_t = 0$ and $0 < E\|\varepsilon_t\|^2 < \infty$. In this paper, \(\varepsilon\) is simply called a strong white noise.

For some \(t_0 \in \mathbb{Z} \cup \{-\infty\}\), let $X=(X_t, t \geq t_0)$ be a stochastic process taking values in $\mathcal B$ satisfying 
\begin{align*}
	X_t = \sum_{j=0}^\infty A_j \varepsilon_{t-j},
\end{align*}
where $(A_j, j\geq 0)$ is a sequence in $\mathcal L_{\mathcal B}$ satisfying $\sum_{j=0}^\infty \|A_j\|_{\mathcal L_{\mathcal B}} < \infty$. We call the sequence $(X_t, t \geq t_0)$ a standard linear process. In this case \(\sum_{j=0}^\infty A_j\) is convergent in $\mathcal L_{\mathcal B}$. 

We say a sequence in $L^2_{\mathcal B}$ is I(0) if it is a standard linear process with \(\sum_{j=0}^\infty A_j \neq 0\).	For \(d \in \{1,2\}\), let \(X = (X_t, t \geq -d+1)\) be a sequence in \(L^2_{\mathcal B}\). We say $(X_t, t \geq 0)$ is I($d$) if its $d$-th differences $\Delta^d X = (\Delta^d X_t, t \geq 1)$ is I(0).

\section{Inversion of a holomorphic $\FF_0$-pencil around an isolated singularity} \label{sholomorphic}
Throughout this section, we employ the following assumption.
\begin{assumption}\label{assumelaurent}
	\(A : U \to \mathcal L_{\mathcal B}\) be a holomorphic Fredholm pencil and  \(z_0 \in \sigma(A)\) is an isolated element.  
\end{assumption}
\noindent Since \(A(z)\) is holomorphic, it allows the Taylor series around \(z_0\) as follows: 	
\begin{align}
	A(z) = \sum_{j=0}^\infty \AZ{j} (z-z_0)^j, \label{powereq}
\end{align}  
where $\AZ{0} = A(z_0)$, $\AZ{j} = {A^{(j)}(z_0)}{/ j !}$ for \(j \geq 1\), and $A^{(j)}(z)$ denotes the $j$-th complex derivative of \(A(z)\). Furthermore, we know from the analytic Fredholm theorem that \( N(z) \coloneqq A(z)^{-1}\) allows the Laurent series expansion in a punctured neighborhood of \(z_0\) as follows:
\begin{align} \label{laurent1}
	N(z) = \sum_{j=-m}^{-1} \NZ{j} (z-z_0)^{j} + \sum_{j=0}^{\infty} \NZ{j} (z-z_0)^{j}, \quad 1 \leq m < \infty. 
\end{align}
Our first goal is to find necessary and sufficient conditions for \(m=1\) and \(2\). We then provide a recursive formula to obtain \(\NZ{j}\) for \(j\geq -m\). Before stating our main assumptions and results of this section, we provide some preliminary results.

First, it can be shown that any Fredholm operator pencil satisfying Assumption \ref{assumelaurent} is in fact an $\FF_0$-pencil. 
\begin{lemma}\label{lem0}
	Under Assumption \ref{assumelaurent}, \(A:U \rightarrow \mathcal L_{\mathcal B}\) is an \(\FF_0\)-pencil.
\end{lemma}
\begin{proof}
	Since \(z_0\) is an isolated element, it implies that there exists some point in \(U\) where the operator pencil is invertible. It turns out that the index of \(A(z)\) does not depend on \(z \in U\) given that \(U\) is connected, and Fredholm operators of nonzero index are not invertible \citep[Section 2]{kaballo2012}. Therefore, this implies that \(A(z)\) has index zero for \(z \in U\).
\end{proof}
\noindent In view of Lemma \ref{lem0}, it may be deduced that the analytic Fredholm theorem provided in Section \ref{sfredholm} is in fact only for $\FF_0$-pencils.    

The following is an important observation implied by Assumption \ref{assumelaurent}.
\begin{lemma}\label{lem1}
	Under Assumption \ref{assumelaurent}, 
	\begin{itemize}
		\item[$\mathrm{(i)}$] \(\ran A(z)\) allows a complementary subspace for \(z \in U\).
		\item[$\mathrm{(ii)}$] \(\ker  A(z)\) allows a complementary subspace for \(z \in U\).
		\item[$\mathrm{(iii)}$] For any finite dimensional subspace \(V\), \(\ran  A(z) + V\) allows a complementary subspace for \(z \in U\)
	\end{itemize}
\end{lemma}
\begin{proof}
We first prove (i).  Since \( A(z)\) is a Fredholm operator, we know that $\ran A(z)$ is closed and \(\mathcal B / \ran  A(z)\) is finite dimensional. Given any closed subspace \(V\), it turns out that $V$ allows a complementary subspace if \(\mathcal B/V\) is finite dimensional  \citep[Theorem 3.2.18]{megginson1998}. Thus, (i) is proved. Our proof of (ii) follows from that every finite dimensional subspace allows a complementary subspace \citep[Theorem 3.2.18]{megginson1998}.  (iii) follows from that the algebraic sum \(\ran  A(z) + V \) is a closed subspace and \(\mathcal B/(\ran  A(z) + V) \)  is finite dimensional  since \(\ran  A(z)\) is closed and \(V\) is finite dimensional.   
\end{proof}
\noindent	  In a Hilbert space, a closed subspace allows a complementary subspace, which can always be chosen as the orthogonal complement. We therefore know that \(\ran A(z)\) and \(\ran A(z) + V\) allow complementary subspaces in a Hilbert space if \(\ran A(z)\) is closed. However in a Banach space, closedness of a subspace does not guarantee the existence of a complementary subspace. The reader is referred to \citet[pp.\ 301-302]{megginson1998} for a detailed discussion on this subject.

\subsection{Simple poles of holomorphic $\FF_0$ inverses}  \label{npre}
Due to Lemma \ref{lem1}, we know that \(\ran \AZ{0}\) and \(\ker \AZ{0}\) are complemented, meaning that we may find their complementary subspaces, as well as the associated bounded projections. Depending on our choice of complementary subspaces, we may also define the corresponding generalized inverse of \(\AZ{0}\) as in \eqref{ginverse}. To simplify expressions, we let 
\begin{align}
	&\indj = \begin{cases}  \id_{\mathcal B} \,\, &\text{ if $j=0$}, \\
		0 \,\, &\text{ otherwise,} \end{cases} \quad\,\,  \notag \\
	&\GZ(\ell,m) = \sum_{k=-m}^{j-1} \NZ{k}\AZ{j+\ell-k}, \quad \ell = 0,1,2,\ldots, \notag \\ 
	&\RR = \ran \AZ{0}, \notag\\
	&\KK = \ker \AZ{0},\notag\\
	&\KKK = \left\{ x \in \KK : \AZ{1}x \in \RR \right\},  \notag\\ 
	&\RC= \text{a complementary subspace of $\ran \AZ{0}$}, \notag	\\
	&\KC= \text{a complementary subspace of $\ker\AZ{0}$},\notag  \\
	&\PR=  \text{the bounded projection onto  $\RC$ along $\RR$}, \notag \\
	&\PK=  \text{the bounded projection onto  $\KC$ along $\KK$}, \notag \\
	&\SZ{1} = \PR \AZ{1}{\mid_{\KK}} : \KK \to \RC,  \notag\\
	&(A_0)^g_{\{\RC,\KC\}}  = \text{the generalized inverse of \(\AZ{0}\)}, \notag
\end{align}
where $\RC$ and $\KC$ depend on our choice. We thus need to be careful with that \(\PR,\PK\) and \(\SZ{1}\) could be differently defined depending on our choice of complementary subspaces; however, given specific choices of \(\RC\) and \(\KC\), they are uniquely defined. Similarly, the subscript \(\{\RC,\KC\}\) of a generalized inverse underscores its dependence on our choice of \(\RC\) and \(\KC\).  

We provide another useful lemma.
\begin{lemma}\label{lem4}
	Suppose that Assumption \ref{assumelaurent} is satisfied. Then invertibility (or noninvertibility) of $\SZ{1}$ does not depend on the choice of $\RC$.
\end{lemma}
\begin{proof}
	Let \(V_0\) and \(W_0\) two different choices of \(\RC\). Then it is trivial to show that 
	\begin{align} \label{lemeq1}
		\ker S_{V_0} = \ker S_{W_0} = \KKK.
	\end{align} 
	Moreover, we know due to Lemma \ref{lem0} that  \(A(z)\) satisfying Assumption \ref{assumelaurent} is in fact an \(\FF_0\)-pencil, which implies that \( \dim (\mathcal B/\ran \AZ{0}) = \dim(\ker \AZ{0}) < \infty\). Since a complementary subspace of \(\ran \AZ{0}\) is isomorphic to \(\mathcal B/\ran \AZ{0}\) \citep[Corollary 3.2.16]{megginson1998}, we have 
	\begin{align} \label{lemeq2}
		\dim(V_0) = \dim(W_0) = \dim(\KK) < \infty.
	\end{align}
	Any injective linear map between finite dimensional vector spaces of the same dimension is also  bijective. Therefore in view of \eqref{lemeq2}, \(\KKK = \{0\}\) is necessary and sufficient condition for \(S_{V_0}\) (and \(S_{W_0}\)) to be invertible. Therefore if either of one is invertible (resp.\ noninvertible), then the other is also invertible (resp.\ noninvertible).  
\end{proof}

We next provide necessary and sufficient conditions for  \(A(z)^{-1}\) to have a simple pole at \(z_0\) and its closed form expression in a punctured neighborhood of \(z_0\).  	


\begin{proposition}\label{propn1}
	Suppose that Assumptions \ref{assumelaurent} are satisfied. Then the following conditions are equivalent to each other. 
	\begin{itemize}
		\item[$\mathrm{(i)}$] $m=1$ in the Laurent series expansion \eqref{laurent1}. 
		\item[$\mathrm{(ii)}$]  $\mathcal B = \RR \oplus \AZ{1} \KK$. 
		\item[$\mathrm{(iii)}$] For all possible choices of $\RC$ ,  $\SZ{1} : \KK \to \RC$ is invertible. 
		\item[$\mathrm{(iv)}$] For some choice of $\RC$ ,  $\SZ{1} : \KK \to \RC$ is invertible. 
	\end{itemize}
	Under any of these conditions and any choice of $\RC$ and $\KC$, the coefficients \((\NZ{j} \geq -1)\) in \eqref{laurent1} are given by the following recursive formula.
	\begin{align}
		&\NZ{-1} =  \SZ{r_1}^{-1} \PZ{r_1}, \label{claimfor0}\\
		&\NZ{j} = 	(\indj- \GZ(0,1)) (A_0)^g_{\{\RC,\KC\}} (\id_{\mathcal B}-\AZ{1}\SZ{r_1}^{-1} \PZ{r_1} )-\GZ(1,1)\SZ{r_1}^{-1} \PZ{r_1}, \label{claimfor}
	\end{align}
	where each $N_j$ is understood as a map from \(\mathcal B\) to \(\mathcal B\) without restriction of the codomain. 
\end{proposition}
\begin{proof} We first show that the claimed equivalence between conditions (i)-(iv), and then verify the recursive formula.\\ 
	\newline		
	\textbf{Equivalence between (i)-(iv)} : Due to the analytic Fredholm theorem, we know that \(A(z)^{-1}\) admits the Laurent series expansion \eqref{laurent1} in a punctured neighborhood \(z_0\). Moreover, \(A(z)\) is holomorphic and thus admits the Taylor series as in \eqref{powereq}.
	Combining \eqref{powereq} and \eqref{laurent1}, we obtain the identity expansion 	\(\id_{\mathcal B} = A(z)^{-1}A(z)\)  as follows.	\begin{align}
		\id_{\mathcal B}&=\sum_{k=-m}^\infty\left(\sum_{j=0}^{m+k}\NZ{k-j}\AZ{j}\right)(z-z_0)^k.\label{idexp}
	\end{align}
	Since (iii) $\Leftrightarrow$ (iv) is deduced from Lemma \ref{lem4}, we demonstrate equivalence between (i)-(iv) by showing (ii)$\Rightarrow$(i)$\Rightarrow$(iv)$\Rightarrow$(ii).
	
	Now we show that (ii)$\Rightarrow$(i). 	Suppose that \(m>1\). Collecting the coefficients of \((z-z_0)^{-m}\) and   \((z-z_0)^{-m+1}\) in \eqref{idexp}, we obtain 		
	\begin{align}
		&\NZ{-m}\AZ{0}=0, \label{N1zero}\\
		&\NZ{-m+1}\AZ{0}+\NZ{-m}\AZ{1}=0. \label{N2zero}
	\end{align}
	Equation \eqref{N1zero} implies that \(	\NZ{-m}\RR = \{0\}\), and further \eqref{N2zero} implies that \(N_{-m}\AZ{1}\KK=\{0\}\). Therefore, if the direct sum decomposition (ii) is true, we necessarily have \(\NZ{-m} = 0\). Note that \(\NZ{-m} = 0\) holds for any \(2 \leq m <\infty \). We therefore conclude that \(m=1\), which proves (ii)$\Rightarrow$(i). 
	
	We next show that (i)$\Rightarrow$(iv). Collecting the coefficients of \((z-z_0)^{-1}\) and   \((z-z_0)^{0}\) in \eqref{idexp} when $m=1$, we have
	\begin{align}
		&\NZ{-1}\AZ{0}=0, \label{N1zero1}\\
		&\NZ{-1}\AZ{1} + \NZ{0}\AZ{0}= \id_{\mathcal B}. \label{N2zero1}
	\end{align}
	Since \(\AZ{0}\) is a Fredholm operator, we know due to Lemma \ref{lem1} that $\RR$ allows a complementary subspace  \(V_0\) and there exists the associated projection operator \( \PP_{V_0} \).  
	Then equation \eqref{N1zero1} implies that 
	\begin{align}
		\NZ{-1}(\id_{\mathcal B}-\PP_{V_0}) = 0 \quad \text{and}\quad  \NZ{-1} = \NZ{-1}\PP_{V_0}. 	 \label{i1eq00}
	\end{align}
	Moreover \eqref{N2zero1} implies $\id_{\mathcal B}{\mid_{\KK}} = \NZ{-1}\AZ{1}{\mid_{\KK}}$. In view of \eqref{i1eq00}, it is apparent that 
	\begin{align}
		\id_{\mathcal B}{\mid_{\KK}} = \NZ{-1}S_{V_0}. \label{i1eq1}
	\end{align}	
	Equation \eqref{i1eq1} implies that  $S_{V_0}$ is an injection. Moreover, due to Lemma \ref{lem0}, we know \(\AZ{0} \in \FF_0\). Using the same arguments we used to establish \eqref{lemeq2}, we obtain
	\begin{align} \label{i1eq2}
		\dim (V_0) = \dim (\mathcal B/\RR)  = \dim(\KK) < \infty.
	\end{align}
	Equations \eqref{i1eq1} and \eqref{i1eq2} together imply that $S_{V_0} : \KK \to   V_0$ is an injective linear map between finite dimensional vector spaces of the same dimension. Therefore, we conclude that $S_{V_0} : \KK \to   V_0$ is a bijection. 
	
	To show (iv)$\Rightarrow$(ii), suppose that our direct sum condition (ii) is false. We  first consider the case where \(\RR \cap \AZ{1}\KK \neq \{0\}\). If there exists a nonzero element $x$ in $\RR \cap \AZ{1}\KK$, we have for any arbitrary choice of $\RC$, \(S_{\RC}x = 0\). This implies that $S_{\RC}$ cannot be injective. We next consider the case where $\mathcal B \neq \RR +  \AZ{1}\KK$ even if \(\RR \cap \AZ{1}\KK = \{0\}\) holds. In this case, clearly \(\RR \oplus  \AZ{1}\KK\) is a strict subspace of \(\mathcal B\). On the other hand, since \(\RC\) is a complementary subspace of \(\RR\), it is deduced that
	\begin{align}
		\dim(\AZ{1}\KK) < \dim(\RC). \label{i1eqq1}
	\end{align}
	Note that \(S_{\RC}\) can be viewed as the composition of \(\PP_{\RC}\) and \(\AZ{1}{\mid_{\KK}}\). 
	From the rank-nullity theorem, \(\dim(S_{\RC}\KK)\) must be at most equal to \(\dim(\AZ{1}\KK)\). In view of \eqref{i1eqq1}, this implies that \(S_{\RC}\) cannot be surjective for any arbitrary choice of $\RC$ . Therefore, we conclude that (iv)$\Rightarrow$(ii).\\
	\newline	
	\textbf{Recursive formula for $(\NZ{j}, j \geq -1)$ } : Assume that  $V_0$ as a choice of \(\RC\) and $W_0$ as a choice of \(\KC\) are fixed. We first verify the claimed formulas \eqref{claimfor0} and \eqref{claimfor} for this specific choice of complementary subspaces. 
	
	At first, we consider the claimed formula for \(\NZ{-1}\). In our demonstration of (i)$\Rightarrow$(iii) above, we obtained \eqref{i1eq1}. Since the codomain of \(S_{V_0}\) is restricted to \(V_0\), \eqref{i1eq1} can be written as 
	\begin{align}\label{i1eq11}
		\id_{\mathcal B}{\mid_{\KK}} = \NZ{-1}{\mid_{V_0}}S_{V_0}.
	\end{align}	
	Moreover we know that \(S_{V_0} : \KK \to {V_0}\) is invertible. We therefore have $\NZ{-1}{\mid_{{V_0}}} = S_{V_0}^{-1}$, where note that we still need to restrict the domain of \(\NZ{-1}\) to ${V_0}$. By composing  both sides of \eqref{i1eq11} with $\PP_{V_0}$, we obtain \(\NZ{-1}\PP_{V_0} = S_{{V_0}}^{-1} \PP_{V_0} \). Recalling \eqref{i1eq00}, which implies that \(\NZ{-1}=\NZ{-1}\PP_{V_0}\), we find that  
	\begin{align}\label{i1eq12}
		\NZ{-1} = S_{{V_0}}^{-1} \PP_{V_0}. 
	\end{align}
	Since the codomain of \(S_{V_0}^{-1}\) is \(K_0\), the map \eqref{i1eq12} is the formula for \(N_{-1}\) with the restricted codomain.  Howeover, it can be understood as a map from \(\mathcal B\) to \(\mathcal B\) by composing both sides of  \eqref{i1eq12}   with a proper embedding. 
	
	Now we verify the recursive formulas for \((\NZ{j}, j \geq 0)\). Collecting the coefficients of $(z-1)^{j}$ and $(z-1)^{j+1}$ in the identity expansion \eqref{idexp}, the following can be shown:  	
	\begin{align}
		&G_j(0,1)
		+ \NZ{j} \AZ{0} = \indj, \label{ideneq1}\\
		&G_j(1,1)
		+  \NZ{j} \AZ{1} + \NZ{j+1} \AZ{0}  = 0. \label{ideneq2}
	\end{align}
	Since \(\id_{\mathcal B} = (\id_{\mathcal B}-\PP_{V_0}) + \PP_{V_0}\), \(\NZ{j}\) can be written as the sum of \(\NZ{j}(\id_{\mathcal B}-P_{V_0})\) and \(\NZ{j}\PP_{V_0}\). We will obtain an explicit formula for each summand.
	
	Given complementary subspaces \({V_0}\) and \(W_0\), we may define \((A_0)^g_{\{V_0,W_0\}} : \mathcal B \to W_0\). Since we have \(\AZ{0}(A_0)^g_{\{V_0,W_0\}} = \id_{\mathcal B}-\PP_{V_0} \), \eqref{ideneq1} implies that 
	\begin{equation} \label{raneq1}
		\NZ{j} (\id_{\mathcal B}-\PP_{V_0}) = 	\indj (A_0)^g_{\{V_0,W_0\}} -G_j(0,1)(A_0)^g_{\{V_0,W_0\}}.
	\end{equation}
	Moreover by restricting the domain of  both sides of \eqref{ideneq2} to $\KK$, we have
	\begin{align}
		G_j(1,1){\mid_{\KK}}
		+  \NZ{j}\AZ{1}{\mid_{\KK}}  = 0.  \label{i1eq13}
	\end{align}
	Since \(\NZ{j} = \NZ{j}\PP_{V_0} + \NZ{j}(\id_{\mathcal B}-\PP_{V_0})\), it is easily deduced  from \eqref{i1eq13} that
	\begin{align}
		\NZ{j}S_{V_0}  = &-G_j(1,1){\mid_{\KK}}- \NZ{j} (\id_{\mathcal B}-\PP_{V_0}) \AZ{1}{\mid_{\KK}}.\label{raneq2}
	\end{align}
	Substituting  \eqref{raneq1} into \eqref{raneq2}, we obtain
	\begin{align}
		\NZ{j}S_{V_0}=&-G_j(1,1){\mid_{\KK}}-\indj (A_0)^g_{\{V_0,W_0\}} \AZ{1}{\mid_{\KK}} + G_j(0,1)
		(A_0)^g_{\{V_0,W_0\}} \AZ{1}{\mid_{\KK}}.  \label{raneq3}
	\end{align}
	Since \(S_{V_0} : \KK \to {V_0} \) is invertible, it is deduced that  \(\ran S_{V_0}^{-1} = \KK\) and $S_{V_0}S_{V_0}^{-1} = \id_{\mathcal B}{\mid_{V_0}}$. We therefore	obtain the following equation from \eqref{raneq3}:
	\begin{align}
		\NZ{j}{\mid_{V_0}}  &= -G_j(1,1)  - \indj (A_0)^g_{\{V_0,W_0\}} \AZ{1}S_{V_0}^{-1} + G_j(0,1)
		(A_0)^g_{\{V_0,W_0\}} \AZ{1}S_{V_0}^{-1}. \label{raneq4}
	\end{align}
	Composing the projection operator of both sides of \eqref{raneq4} with \(\PP_{V_0}\), we then obtain an explicit formula for \(\NZ{j}\PP_{V_0}\). Combining this result with  \eqref{raneq1}, we obtain the formula \(\NZ{j}\) similar to \eqref{claimfor} in terms of \(\PP_{V_0}\), \((A_0)^g_{\{V_0,W_0\}}\), \(S_{V_0}\), \(G_j(0,1)\), and \(G_j(1,1)\)  after a little algebra. Of course, the resulting operator \(N_j\) should be understood as a map from \(\mathcal B\) to \(\mathcal B\).
	
	Our formula for each \(N_j\) that we have obtained seems to depend on our choice of complementary subspaces, especially due to \(\PP_{V_0}\), \(S_{V_0}\) and \((A_0)^g_{\{V_0,W_0\}}\). However, if a Laurent series exists, it is unique. We could differently define the aforementioned operators by choosing different complementary subspaces, and then could obtain a recursive formula for $(N_j, j \geq -1)$ in terms of those operators. However, such a newly obtained formula cannot be different from what we have obtained from a fixed choice of complementary subspaces due to the uniqueness of the Laurent series. Therefore, it is easily deduced that our recursive formula for \(N_j\) derived in Proposition \ref{propn1} does not depend on a specific choice of complementary subspaces.  
\end{proof}

\subsection{Second order poles of holomorphic $\FF_0$ inverses}
To simplify expressions, we let 
\begin{align*}
	&\RRR = \ran \AZ{0} + \AZ{1} \ker \AZ{0},\notag\\
	&\RCC = \text{a complementary subspace of $\ran \AZ{0} + \AZ{1} \ker \AZ{0}$},\notag\\
	&\KCC = \text{a complementary subspace of $\KKK$ in $\KK$}, \notag	\\
	&\PRR=  \text{the bounded projection onto  $\RCC$ along $\RRR$}, \notag \\
	&\PKK=  \text{the bounded projection onto  $\KCC$ along $\KKK$}. \notag 
\end{align*}
We know from Lemma \ref{lem1} that \(\RR\), \( \KK \) and \(\RRR\) are complemented, so we may find complementary subspaces \( \RC \), \( \KC\), and \(\RCC\), as well as the bounded projections \(\PP_{\RC}\), \(\PP_{\KC}\) and \(\PP_{\RCC}\). Given \(\RC\), $\RCC$ is not uniquely determined in general. We require our choice to satisfy
\begin{align} \label{choice}
	\RCC  \subset \RC, 
\end{align}     
so that
\begin{align}
	& \RC = \SZ{1}  \KK \oplus \RCC, \label{subdirecteq}\\
	&\PZ{r_1}\PP_{\RCC} = \PP_{\RCC}\PP_{\RC}= \PP_{\RCC}. \label{projeceq}
\end{align} 
Given $\RC$, a choice of a complementary subspace satisfying \eqref{choice} is always possible, and such a subspace is easily obtained as follows: 
\begin{lemma}\label{lemchoice} Suppose that Assumption \ref{assumelaurent} is satisfied.
	Given $\RC$, let \(V_1\) be a specific choice of $\RCC$. Then \(\PZ{1}V_1 \subset \RC\) is also a complementary subspace of $\RRR$.
\end{lemma}
\begin{proof}
	Let \(V_0\) be a given choice of \({\RC}\).		
	If $\mathcal B = \RR + \AZ{1} \KK $, then \(V_1 = \{0\}\), then our statement trivially holds. Now consider the case when \(V_1\) is a nontrivial subspace. 
	Since we have $\RR + \AZ{1}\KK = \RR \oplus \PP_{V_0}\AZ{1}\KK$ holds, it is  deduced that $\mathcal B = \RR \oplus \PP_{V_0}\AZ{1}\KK \oplus V_1 $. This implies that \(M \coloneqq \PP_{V_0}\AZ{1}\KK \oplus V_1\) is a complementary subspace of \(\RR\). Since  \(\PP_{V_0}\mathcal B = \PP_{V_0}M\), clearly \(\PP_{V_0}{\mid_{M}} : M \to  V_0\) must be a surjection, so we have 
	\begin{align}
		\PP_{V_0} M = V_0. \label{lemeq00}
	\end{align}
	Moreover, both \(M\) and \( V_0\) are complementary subspaces of \(\RR\), and we know due to Lemma \ref{lem0} that \(\AZ{0} \in \FF_0\). Then it is deduced from  similar arguments to those we used to derive \eqref{lemeq2} that 
	\begin{align*}
		\dim(\mathcal B/\RR) = \dim(V_0) = \dim(M).
	\end{align*}
	Thus, \(\PP_{V_0}{\mid_{M}} : M \to  \RC\) is a surjection between vector spaces of the same finite dimension, meaning that it is also an injection. We therefore obtain  \(\PP_{V_0} \AZ{1}\KK \cap \PP_{V_0} V_1 = \{0\}\), which implies that \(\PP_{V_0}M = \PP_{V_0} \AZ{1}\KK \oplus  \PP_{V_0} V_1\). Combining this with \eqref{lemeq00}, it is deduced that
	\begin{align*}
		\mathcal B = \RR  \oplus \PP_{V_0} \AZ{1}\KK \oplus  \PP_{V_0} V_1.
	\end{align*}
	Clearly \(\PP_{V_0} V_1 \) is a complementary subspace of \(\RRR\).
\end{proof}
\noindent Due to Lemma \ref{lemchoice}, we know how to make an arbitrary choice of \(\RCC \) satisfy the requirement \eqref{choice} and thus may assume that our choice of \(\RCC\) satisfies \eqref{choice} in the subsequent discussion.

Under any choice of our complementary subspaces satisfying \eqref{choice}, we define
\begin{align*}
	&A^{\dagger}_{2\,\{\RC, \KC\}} = \AZ{2} - \AZ{1}(A_0)^g_{\{\RC, \KC\}}\AZ{1}, \\
	&S^\dagger_{\{\RC, \KC, \RCC\}} = \PP_{\RCC}A^{\dagger}_{2\,\{\RC, \KC\}}{\mid_{\KKK}} : \KKK \to \RCC,
\end{align*} 
where subscripts also indicate the collection of complementary subspaces upon which the above operators depend. 

In this section, we consider the case \(\KKK \neq \{0\}\).  Then \(\SZ{r_1}\) is not invertible since \(\ker \SZ{r_1} = \KKK\). However, note that \(\SZ{r_1}\) is a linear map between finite dimensional subspaces, so we can always define its generalized inverse as follows:
\begin{align}
	\SG = \left(\SZ{r_1}{\mid_{\KCC}} \right)^{-1}(\id_{\mathcal B}-\PP_{\RCC}){\mid_{\RC}}.\label{sgeneral}
\end{align}

Before stating our main proposition of this section, we first establish the following preliminary result.
\begin{lemma}\label{lem5}
	Suppose that Assumption \ref{assumelaurent} is satisfied. Let \(V_0\) and \(\widetilde{V}_0\) be arbitrary choices of \(\RC\), and
	\(V_1 \subset V_0\) and \(\widetilde{V}_1 \subset \widetilde{V}_0\) be arbitrary choices of \(\RCC\). Then 
	\begin{align*}
		\dim(V_1) = \dim(\widetilde{V_1}) = \dim(\KKK).	
	\end{align*}
\end{lemma}
\begin{proof}
	For \(V_0\) and \(\widetilde{V_0}\), we have two defined operators \(S_{V_0} : \KK \to V_0\) and \(S_{\widetilde{V}_0}:\KK \to \widetilde{V}_0\). We established that \(\ker S_{V_0}=\ker S_{\widetilde{V}_0} = \KKK\) in \eqref{lemeq1}. 
	From Lemma \ref{lem0}, we know \(\AZ{0} \in \FF_0 \), so it is easily deduced that
	\begin{align}
		&\dim(V_0) = \dim(\KK) = \dim(S_{V_0}\KK) + \dim(\KKK),\label{lemeq3}  \\ 
		&\dim(\widetilde{V}_0) = \dim(\KK) = \dim(S_{\widetilde{V}_0}\KK) + \dim(\KKK). \label{lemeq4}  
	\end{align}
	In each of \eqref{lemeq3} and \eqref{lemeq4}, the first equality is deduced from the same argument to those we used to derive \eqref{lemeq2}, and the second equality is justified by the rank-nullity theorem. 
	Moreover, the following  direct sum decompositions are allowed:
	\begin{align}
		V_0  &= S_{V_0}\KK  \oplus V_1, \label{lemeq5} \\ 
		\widetilde{V}_0  &= S_{\widetilde{V}_0}\KK  \oplus \widetilde{V}_1.\label{lemeq6}
	\end{align}
	To see why \eqref{lemeq5} and \eqref{lemeq6}  are true, first note that we have \(\RR + \AZ{1}\KK = \RR \oplus  S_{V_0}\KK = \RR \oplus  S_{\widetilde{V}_0}\KK\). We thus have 
	\(\mathcal B = \RR \oplus  S_{V_0}\KK \oplus V_1=\RR \oplus  S_{\widetilde{V}_0}\KK \oplus \widetilde{V}_1\). These direct sum conditions imply that \(S_{V_0}\KK \oplus V_1\) and \(S_{\widetilde{V}_0}\KK \oplus \widetilde{V}_1\) are  complementary subspaces of \(\RR\). Since $V_1 \subset V_0$ and $\widetilde{V}_1 \subset \widetilde{V}_0$, \eqref{lemeq5} and \eqref{lemeq6} are established. Now it is deduced from \eqref{lemeq5} and \eqref{lemeq6} that 
	\begin{align}
		\dim(V_0) &= \dim(S_{V_0}\KK) + \dim(V_1),  \label{lemeq7} \\ 
		\dim(\widetilde{V}_0) &= \dim(S_{\widetilde{V}_0}\KK) + \dim(\widetilde{V}_1). \label{lemeq8}
	\end{align}
	Comparing \eqref{lemeq3} and \eqref{lemeq7}, we obtain \(\dim(\KKK) = \dim(V_1)\). Additionally from \eqref{lemeq4} and \eqref{lemeq8}, we obtain \(\dim(\KKK) = \dim (\widetilde{V}_1)\).
\end{proof}	


Now we provide necessary and sufficient conditions for \(A(z)^{-1}\) to have a second order pole at $z_0$ and its closed-form expression in a punctured neighborhood of $z_0$. 

\begin{proposition}\label{propn2}
	Suppose that Assumptions \ref{assumelaurent} are satisfied and \(\KKK \neq \{0\}\). Then the following conditions are equivalent to each other. 
	\begin{itemize}
		\item[$\mathrm{(i)}$] $m=2$ in the Laurent series expansion \eqref{laurent1}. 
		\item[$\mathrm{(ii)}$]  For some choice of $\RC$, $\KC$, we have 
		\begin{align*}
			\mathcal B = \RRR \oplus A^{\dagger}_{2\,\{\RC, \KC\}} \KKK.
		\end{align*}
		\item[$\mathrm{(iii)}$]   For all possible choices of $\RC$, $\KC$, and $\RCC$ satisfying \eqref{choice}, $S^\dag_{\{\RC, \KC, \RCC\}} : \KKK \to \RCC \text{ is invertible. } $
		\item[$\mathrm{(iv)}$]   For some choice of $\RC$, $\KC$, and $\RCC$ satisfying \eqref{choice},
		$S^\dag_{\{\RC, \KC, \RCC\}} : \KKK \to \RCC \text{ is invertible. } $
	\end{itemize}
	Under any of these conditions and any choice of complementary subspaces satisfying \eqref{choice}, the coefficients \((\NZ{j} \geq -2)\) in \eqref{laurent1} are given by the following recursive formula.
	\begin{align}
		\NZ{-2}  &=  (S^\dagger_{\{\RC, \KC, \RCC\}})^{-1} \PP_{\RCC}, \label{n2claim} \\
		\NZ{-1}  &= \left(\QR\SG \PP_{\RC} - \NZ{-2} \AZ{1} (\AZ{0})^g_{\{\RC,\KC\}}  \right)\QL \nonumber\\ 
		\quad\quad\quad\quad & \quad-\QR(\AZ{0})^g_{\{\RC,\KC\}}  \AZ{1}\NZ{-2}
		- \NZ{-2} A^{\dagger}_{3\,\{\RC,\KC\}}   \NZ{-2}, \label{n1claim}\\
		\NZ{j} &= \left( G_j(1,2)(\AZ{0})^g_{\{\RC,\KC\}}\AZ{1} - G_j(2,2)\right)\NZ{-2}  \notag \\ 
		&\quad+ \left(\indj-G_j(0,2)\right)(\AZ{0})^g_{\{\RC,\KC\}}\left(\id_{\mathcal B}-A_1\SG \PP_{\RC}\right)\QL  \notag \\
		&\quad- G_j(1,2)\SG\PP_{\RC}\QL.  \label{njclaim}
	\end{align} 
	\normalsize
	
	where, 	\begin{align*}
		A^{\dagger}_{3\,\{\RC,\KC\}} &=  \AZ{3}-\AZ{1}(\AZ{0})^g_{\{\RC,\KC\}}\AZ{1}(\AZ{0})^g_{\{\RC,\KC\}}\AZ{1},\\
		\QL &= \id_{\mathcal B} - A^{\dagger}_{2\,\{\RC,\KC\}}\NZ{-2},\\ 
		\QR  &= \id_{\mathcal B} - \NZ{-2} A^{\dagger}_{2\,\{\RC,\KC\}}.
	\end{align*}
	Each $N_j$  is understood as a map from \(\mathcal B\) to \(\mathcal B\) without restriction of the codomain. 	
\end{proposition}

\begin{proof}
	We first establish some results are repeatedly mentioned in the subsequent proof. Given any choice of complementary subspaces satisfying \eqref{choice}, the following identity decomposition is easily deduced from  \eqref{projeceq}:
	\begin{align}\label{iddecomp}
		\id_{\mathcal B} = ( \id_{\mathcal B} - \PP_{\RC}) + (\id_{\mathcal B} - \PP_{\RCC})\PP_{\RC} + \PP_{\RCC}.
	\end{align}
	Since we have \(\RRR = \RR + \AZ{1} \KK =  \RR \oplus \SZ{r_1}\KK\), our direct sum condition (ii) is  equivalent to  
	\begin{align}\label{directsum22}
		\mathcal B =   \RR \oplus \SZ{r_1}\KK \oplus A^\dagger_{2\,\{\RC,\KC\}} \KKK.
	\end{align} 
	Moreover, we may obtain the following expansion of the identity from \eqref{powereq} and \eqref{laurent1}:
	\begin{align}
		\id_{\mathcal B} &=\sum_{k=-m}^\infty\left(\sum_{j=0}^{m+k}\NZ{k-j}\AZ{j}\right)(z-z_0)^k\label{idexp2} \\
		&=\sum_{k=-m}^\infty\left(\sum_{j=0}^{m+k}\AZ{j}\NZ{k-j}\right)(z-z_0)^k.\label{idexp3}	
	\end{align}\\


\textbf{Equivalence between (i)-(iv)} :
Since (iii) $\Rightarrow$ (iv) is trivial, we will show that (ii)$\Rightarrow$(i)$\Rightarrow$(iii) and (iv) $\Rightarrow$ (ii).

To show (ii)\(\Rightarrow\)(i), let \(V_0\) (resp.\ $W_0$) be a choice of \(\RC\) (resp.\ $\KC$), and the direct sum condition (ii) holds for \(V_0\) and \(W_0\). Since \(\ker S_{V_0} = \KKK \neq \{0\}\), \(S_{V_0}\) cannot be invertible. Therefore, $m \neq 1$ by Proposition \ref{propn1}. Therefore, suppose that $2 \leq m < \infty$ in \eqref{laurent1}. Collecting the coefficients of \((z-z_0)^{-m}\), \((z-z_0)^{-m+1}\)  and \((z-z_0)^{-m+2}\) in \eqref{idexp2} and \eqref{idexp3}, we obtain 
\begin{align} 
	&\NZ{-m} \AZ{0} = \AZ{0}\NZ{-m}  = 0, \label{i2eq1}\\
	&\NZ{-m} \AZ{1} + \NZ{-m+1} \AZ{0} = \AZ{1}\NZ{-m}    + \AZ{0} \NZ{-m+1} 
	= 0, \label{i2eq2} \\
	&\NZ{-m} \AZ{2}  + \NZ{-m+1}\AZ{1}  +  \NZ{-m+2} \AZ{0}  = 0. \label{i2eq3}
\end{align}  	
We may define the generalized inverse \((A_0)^g_{\{V_0, W_0\}}\) for \(V_0\) and \(W_0\).  
Composing both sides of  \eqref{i2eq1} with \((A_0)^g_{\{V_0, W_0\}}\), we obtain
\begin{align} \label{i2eq11}
	\NZ{-m}(\id_{\mathcal B}-\PP_{V_0}) = 0 \quad \text{and} \quad \NZ{-m} = \NZ{-m} \PP_{V_0}.
\end{align}
From  \eqref{i2eq2} and \eqref{i2eq11}, it is deduced that	
\begin{equation} \label{i2eq4}
	\NZ{-m} \AZ{1}{\mid_{\KK}} = \NZ{-m} \PP_{V_0} \AZ{1}\mid_{\KK} =	N_{-m} S_{V_0} = 0.
\end{equation}
Restricting the domain of the both sides of \eqref{i2eq3} to $\KKK$, we find that
\begin{equation} \label{i2eq6}
	\NZ{-m}  \AZ{2}{\mid_{\KKK}} + \NZ{-m+1} \AZ{1}{\mid_{\KKK}}  = 0.
\end{equation}
Moreover \eqref{i2eq2} trivially implies that  
\begin{equation}
	\NZ{-m+1} \AZ{0} = - \NZ{-m} \AZ{1} \:\:\: \text{and}\:\:\:   \AZ{0}\NZ{-m+1} = - \AZ{1}\NZ{-m}.   \label{i2eqq001}
\end{equation}
By composing each of \eqref{i2eqq001} with $(A_0)^g_{\{V_0, W_0\}}$,  it can be deduced that 
\begin{align}
	&\NZ{-m+1}(\id_{\mathcal B}-\PP_{V_0}) = - \NZ{-m} \AZ{1} (A_0)^g_{\{V_0, W_0\}},   \label{nmeq1}\\
	&\PP_{W_0}\NZ{-m+1} 	= - (A_0)^g_{\{V_0, W_0\}}  \AZ{1} \NZ{-m}.   \label{n1partial1}
\end{align} 
Composing  both sides of \eqref{n1partial1} with  \(\PP_{V_0}\) and using \eqref{i2eq11}, we find  that   
\begin{equation} \label{i2pfeq01}
	\PP_{W_0}\NZ{-m+1}  \PP_{V_0} = -(A_0)^g_{\{V_0, W_0\}} \AZ{1} \NZ{-m}.
\end{equation}
From \eqref{i2pfeq01} and the identity decomposition \(\id_{\mathcal B} = (\id_{\mathcal B}-\PP_{W_0}) + \PP_{W_0}\), we obtain
\begin{equation} \label{nmeq2}
	\NZ{-m+1} \PP_{V_0} = -(A_0)^g_{\{V_0, W_0\}} \AZ{1} \NZ{-m}  +   (\id_{\mathcal B} - \PP_{W_0})\NZ{-m+1}  \PP_{V_0}.
\end{equation}
Summing both sides of \eqref{nmeq1} and \eqref{nmeq2} gives 
\begin{align} 
	\NZ{-m+1} =  &-\NZ{-m} \AZ{1}  (A_0)^g_{\{V_0, W_0\}}      - (A_0)^g_{\{V_0, W_0\}}  \AZ{1} \NZ{-m} + (\id_{\mathcal B} - \PP_{W_0})\NZ{-m+1}  \PP_{V_0}. \label{i2eq7}
\end{align}
Therefore, \eqref{i2eq6} and \eqref{i2eq7} together imply that
\begin{align}
	0 \,= \,& \NZ{-m}  \AZ{2}{\mid_{\KKK}} - \NZ{-m} \AZ{1} (A_0)^g_{\{V_0, W_0\}}\AZ{1}{\mid_{\KKK}}  -(A_0)^g_{\{V_0, W_0\}}  \AZ{1} \NZ{-m}\AZ{1}{\mid_{\KKK}} \notag \\  &+  (\id_{\mathcal B} - \PP_{W_0})\NZ{-m+1}  \PP_{V_0} \AZ{1}{\mid_{\KKK}}. \label{i2eq8}
\end{align} 
From the definition of \(\KKK\), \( \PP_{V_0} \AZ{1}{\mid_{\KKK}} = 0\). Therefore, the last term  in \eqref{i2eq8} is zero . Moreover in view of \eqref{i2eq11}, we have $\NZ{-m}(\id_{\mathcal B}-\PP_{V_0}) = 0$. This implies that the third term in \eqref{i2eq8} is zero and \eqref{i2eq8} reduces to 
\begin{align} \label{i2eq9}
	\NZ{-m}  A^\dagger_{2\,\{V_0, W_0\}}{\mid_{\KKK}} = 0. 
\end{align}  
Given our direct sum condition (ii) (or equivalently \eqref{directsum22}) with equations \eqref{i2eq11}, \eqref{i2eq4}  and \eqref{i2eq9}, we conclude that \(\NZ{-m} = 0\). The above arguments hold for any arbitrary choice of \(m\) such that \(2<m<\infty\), and we already showed that \(m=1\) is impossible. Therefore, $m$ must be 2. This proves (ii)\( \Rightarrow \)(i). 

Now we show that (i)\(\Rightarrow\)(iii). We let \(V_0\), \(W_0\), and \(V_1 (\subset V_0)\) be choices of \(\RC, \KC\) and \(\RCC\) respectively. Suppose that  \(S^\dagger_{\{V_0, W_0,V_1\}} \) is not invertible. Due to Lemma \ref{lem5}, we know \(\dim(V_1) = \dim(\KKK)\), meaning that \(S^\dag_{\{V_0, W_0,V_1\}}\) is not injective. Therefore, we know there exists an element \(x \in \KKK\) such that \(S^\dag_{\{V_0, W_0,V_1\}} x = 0\). Collecting the coefficients of \((z-z_0)^{-1}\), \((z-z_0)^{0}\)  and \((z-z_0)^{0}\) in \eqref{idexp2} and \eqref{idexp3},  we  have
\begin{align} 
	&\sum_{k=-m}^{-3} \NZ{k}\AZ{-2-k} + \NZ{-2} \AZ{0} = 0, \label{i2eq0000}\\
	&\sum_{k=-m}^{-3} \NZ{k}\AZ{-1-k} + \NZ{-2} \AZ{1} + \NZ{-1}\AZ{0} = 0, \label{i2eq1111}\\
	&\sum_{k=-m}^{-3} \NZ{k}\AZ{-k} + \NZ{-2} \AZ{2} + \NZ{-1}\AZ{1} +  \NZ{0}\AZ{0}= \id_{\mathcal B}. \label{i2eq2222} 
\end{align} 
From \eqref{iddecomp}, \(N_{-2}\) can be written as the sum of  \(\NZ{-2}(\id_{\mathcal B}- \PP_{V_0})\), \(\NZ{-2}(\id_{\mathcal B} - \PP_{V_1}) \PP_{V_0}\) and \(\NZ{-2}\PP_{V_1}\). We will obtain an explicit formula for each summand. It is deduced from \eqref{i2eq0000} that 
\begin{align}
	\NZ{-2} (\id_{\mathcal B}- \PP_{V_0}) = - \sum_{k=-m}^{-3} \NZ{k}\AZ{-2-k}(A_0)^g_{\{V_0, W_0\}}. \label{i2eq4444}
\end{align}
Restricting both sides of \(\eqref{i2eq1111}\) to \(\KK\), we obtain
\begin{align}
	\NZ{-2} \AZ{1}{\mid_{\KK}}  = -  \sum_{k=-m}^{-3} \NZ{k}\AZ{-1-k}{\mid_{\KK}}. \label{i2eq5555}
\end{align} 
Since \(\NZ{-2} = \NZ{-2}(\id_{\mathcal B} - \PP_{V_0}) + \PP_{V_0}\), we obtain from \eqref{i2eq4444} and \eqref{i2eq5555},
\begin{align}
	\NZ{-2} S_{V_0} = -  \sum_{k=-m}^{-3} \NZ{k}\AZ{-1-k}{\mid_{\KK}} + \sum_{k=-m}^{-3} \NZ{k}\AZ{-2-k}(A_0)^g_{\{V_0, W_0\}} A_1{\mid_{\KK}}. \label{i2eq6666}
\end{align}  
We may define $\SSG$ as in \eqref{sgeneral}. Composing both sides of \eqref{i2eq6666} with  $\SSG \PP_V $, then we obtain
\begin{align}
	\NZ{-2}(\id_{\mathcal B} - \PP_{V_1}) \PP_{V_0}  = &-  \sum_{k=-m}^{-3} \NZ{k}\AZ{-1-k}\SSG \PP_{V_0} \notag \\ &+ \sum_{k=-m}^{-3} \NZ{k}\AZ{-2-k}(\AZ{0})^g_{\{V_0,W_0\}} A_1\SSG \PP_{V_0}.  \label{i2eq7777}
\end{align}  
Restricting both sides of \(\eqref{i2eq2222}\) to \(\KKK\), we have
\begin{align}
	&\sum_{k=-m}^{-3} \NZ{k}\AZ{-k}{\mid_{\KKK}} + \NZ{-2} \AZ{2}{\mid_{\KKK}} +  \NZ{-1}\AZ{1}{\mid_{\KKK}} = \id_{\mathcal B}{\mid_{\KKK}}. \label{i2eq22222} 
\end{align}	
From \eqref{i2eq1111}, we can also obtain
\begin{align} 
	&\NZ{-1}(\id_{\mathcal B}-\PP_{V_0}) = -\sum_{k=-m}^{-3} \NZ{k}\AZ{-1-k}(A_0)^g_{\{V_0, W_0\}} -\NZ{-2} \AZ{1} (A_0)^g_{\{V_0, W_0\}}. \label{i2eq3333}
\end{align} 
Since $\AZ{1}\KKK \subset \RR$, we have $ \NZ{-1}\AZ{1}{\mid_{\KKK}} =  \NZ{-1}(\id_{\mathcal B}-\PP_{V_0})\AZ{1}{\mid_{\KKK}}$. Substituting \eqref{i2eq3333} into  \eqref{i2eq22222}, the following can be obtained:
\begin{align*}
	\sum_{k=-m}^{-3} \NZ{k}\AZ{-k}{\mid_{\KKK}} -\sum_{k=-m}^{-3} \NZ{k}\AZ{-1-k}(A_0)^g_{\{V_0, W_0\}}\AZ{1}{\mid_{\KKK}} + \NZ{-2}A^\dagger_{2\,\{V_0, W_0\}}{\mid_{\KKK}} = \id_{\mathcal B}{\mid_{\KKK}}. 
\end{align*}
Since \(\NZ{-2} = \NZ{-2}(\id_{\mathcal B}-\PP_{V_0}) + \NZ{-2}(\id_{\mathcal B}-\PP_{V_1})\PP_{V_0} + \NZ{-2}\PP_{V_1} \), we have
\begin{align}
	\id_{\mathcal B}{\mid_{\KKK}}  =& \sum_{k=-m}^{-3} \NZ{k}\AZ{-k}{\mid_{\KKK}} -\sum_{k=-m}^{-3} \NZ{k}\AZ{-1-k}(A_0)^g_{\{V_0, W_0\}}\AZ{1}{\mid_{\KKK}} \notag \\ &+\NZ{-2}(\id_{\mathcal B}-\PP_{V_0})A^\dagger_{2\,\{V_0, W_0\}}{\mid_{\KKK}}  + \NZ{-2}(\id_{\mathcal B}-\PP_{V_1})\PP_{V_0}A^\dagger_{2\,\{V_0, W_0\}}{\mid_{\KKK}} +\NZ{-2}S^\dag_{\{V_0, W_0,V_1\}}. \label{i2eq11111}
\end{align}
Note that if \(\NZ{j}\) is zero for every \(j \leq -3\), the first four terms of the right hand side of \eqref{i2eq11111} are equal to zero, which can be easily deduced from the obtained formulas  for $ \NZ{-2}(\id_{\mathcal B}-\PP_{V_0})$ and \(\NZ{-2}(\id_{\mathcal B}-\PP_{V_1})\PP_{V_0}\) in \eqref{i2eq4444} and \eqref{i2eq7777}. However, we showed that there exists some \(x \in \KKK\) such that \(S^\dag_{\{V_0, W_0,V_1\}}x = 0\), which implies that \(\NZ{j}\) for some \(j \leq -3\) must not be zero. This shows (i)\(\Rightarrow\)(iii).

It remains to show (iv) $\Rightarrow$ (ii). Suppose that (ii) does not hold. Then for any arbitrary choice of $\RC$ and $\KC$, we must have either of
\begin{align}
	\RRR \cap A^\dagger_{2\,\{\RC, \KC\}}\KKK \neq \{0\}  \label{eqadd01}
\end{align}
or 
\begin{align}
	\RRR + A^\dagger_{2\,\{\RC, \KC\}}\KKK \neq \mathcal B. \label{eqadd02}
\end{align}
If \eqref{eqadd01} is true, then clearly  \(S^\dag_{\{\RC, \KC,\RCC\}}\) cannot be injective for any arbitrary choice of $\RCC$ satisfying \eqref{choice}.  Moreover if \eqref{eqadd02} is true, then we must have $\dim( A^\dagger_{2\,\{\RC, \KC\}}\KKK) < \dim(\RCC)$. This implies that \(S^\dag_{\{\RC, \KC,\RCC\}}\) cannot be surjective for any arbitrary choice of $\RCC$ satisfying \eqref{choice}. Therefore  (iv) $\Rightarrow$ (ii) is easily deduced.
\\   
\newline	
\textbf{Formulas for \(\NZ{-2}\) and \(\NZ{-1}\)} : We let \(V_0\), \(W_0\), \(V_1 (\subset V_0)\), \(W_1\) our  choice of \(\RC, \KC\), \(\RCC\) and \(\KCC\), respectively. Collecting the coefficients of \((z-z_0)^{-2}, (z-z_0)^{-1}\) and \((z-z_0)^{0}\) from \eqref{idexp2} and \eqref{idexp3}, we have
\begin{align*} 
	&\NZ{-2} \AZ{0} = \AZ{0}\NZ{-2}  = 0,\\
	&\NZ{-2} \AZ{1} + \NZ{-1} \AZ{0} = \AZ{1}\NZ{-2}    + \AZ{0} \NZ{-1} 
	= 0, \\
	&\NZ{-2} \AZ{2}  + \NZ{-1}\AZ{1}  +  \NZ{0} \AZ{0}  = 0.
\end{align*}  	
From similar arguments and algebra to those in our demonstration of (ii)$\Rightarrow$(i), it can easily be deduced that
\begin{align}
	&\NZ{-2} \RRR = \{0\}, \label{i2eqq4}\\
	&\NZ{-2} A^\dag_{2\,\{V_0,W_0\}}{\mid_{\KKK}} = \id_{\mathcal B}{\mid_{\KKK}}. \label{i2eqq5}
\end{align}
Equation \eqref{i2eqq4} implies that 
\begin{align}
	\NZ{-2} (\id_{\mathcal B}-\PP_{V_1}) = 0  \quad \text{and} \quad \NZ{-2} = \NZ{-2} \PP_{V_1}.\label{i2eqq55}
\end{align}
Equations \eqref{i2eqq4} and \eqref{i2eqq5} together imply that 
\begin{align}
	\NZ{-2}{\mid_{V_1}}S^\dagger_{\{V_0,W_0,V_1\}} = \id_{\mathcal B}{\mid_{\KKK}}. \label{i2eqq6}
\end{align}
Composing both sides of \eqref{i2eqq6} with \((S^\dagger_{\{V_0,W_0,V_1\}})^{-1} \PP_{V_1}\), we obtain 
\begin{align} \label{i2eqq7}
	\NZ{-2}\PP_{V_1} = (S^\dagger_{\{V_0,W_0,V_1\}})^{-1} \PP_{V_1}.
\end{align} 
In view of \eqref{i2eqq55}, \eqref{i2eqq7} is in fact equal to \(N_{-2}\) with the codomain restricted to \(\mathcal K_1\). Viewing this as a map from \(\mathcal B\) to \(\mathcal B\), we obtain \eqref{n2claim} for our choice of complementary subspaces. 

We next verify the claimed formula for \(\NZ{-1}\). In view of the identity decomposition \eqref{iddecomp}, $\NZ{-1}$ may be written as the sum of \(\NZ{-1}(\id_{\mathcal B} - \PP_{V_0})\), \(\NZ{-1}(\id_{\mathcal B} - \PP_{V_1})\PP_{V_0}\) and \(\NZ{-1}\PP_{V_1}\). We will find an explicit formula for each summand. From \eqref{idexp2} when \(m=2\), we obtain the coefficients of  $(z-z_0)^{-1}$, $(z-z_0)^{0}$ and $(z-z_0)^{1}$  as follows.
\begin{align} 
	&\NZ{-2} \AZ{1} + \NZ{-1} \AZ{0} 
	= 0, \label{i2n12} \\
	&\NZ{-2} \AZ{2}  + \NZ{-1}\AZ{1}  +  \NZ{0} \AZ{0}  = \id_{\mathcal B}, \label{i2n13}\\
	&\NZ{-2} \AZ{3} + \NZ{-1} \AZ{2} + \NZ{0} \AZ{1} + \NZ{1} \AZ{0} = 0. \label{i2n14}
\end{align} 
From (\ref{i2n12}) and the properties of the generalized inverse, it is easily deduced that
\begin{equation} \label{i2n15}
	\NZ{-1}  (\id_\mathcal B-\PP_{V_0}) = - \NZ{-2}\AZ{1} (\AZ{0})^g_{\{V_0,W_0\}}. 
\end{equation}
Restricting the domain of the both sides of \eqref{i2n13} to \( \KK\), we obtain 
\begin{equation}\label{i2n155}
	\NZ{-1}\AZ{1}{\mid_{ \KK}} = \id_{\mathcal B}{\mid_{ \KK}} - \NZ{-2} \AZ{2}{\mid_{ \KK}}. 
\end{equation}
Using the identity decomposition $\id_{\mathcal B} = \PP_{V_0}+(\id_{\mathcal B}-\PP_{V_0})$, \eqref{i2n155} can be written as
\begin{align}
	\NZ{-1}S_{V_0}  = \id_{\mathcal B}{\mid_{\KK}} - \NZ{-2} \AZ{2}{\mid_{\KK}}  -  \NZ{-1} &(\id_{\mathcal B}-\PP_{V_0}) \AZ{1}{\mid_{\KK}}. \label{i2n1555}
\end{align}
Substituting \eqref{i2n15} into \eqref{i2n1555}, we obtain
\begin{align} \label{i2n133}
	&\NZ{-1}S_{V_0} = \left(\id_{\mathcal B} - \NZ{-2} A^\dagger_{2\,\{V_0,W_0\}}\right){\mid_{\KK}}. 
\end{align}
Under our direct sum condition (ii), $S_{V_0}: \KK \to V_0$ is not invertible, but allows a generalized sense as in \eqref{sgeneral}. From the construction of \(\SSG\), we have  \(S_{V_0} \SSG = (\id_\mathcal B-\PP_{V_1}){\mid_{V_0}}\). 	Composing both sides of \eqref{i2n133} with \(\SSG\PP_{V_0}\), we obtain
\begin{align}\label{i2n151}
	\NZ{-1} (\id_\mathcal B-\PP_{V_1})\PP_{V_0} = \left(\id_{\mathcal B} - \NZ{-2} A^\dagger_{2\,\{V_0,W_0\}}\right)\SSG\PP_{V_0}. 
\end{align} 
Restricting the domain of both sides of \eqref{i2n14} to \(\KKK\), we have
\begin{equation} \label{i2n16}
	\NZ{-2} \AZ{3}{\mid_{\KKK}} + \NZ{-2} \AZ{2} + \NZ{0} \AZ{1}{\mid_{\KKK}} = 0. 
\end{equation}
Composing both sides of  \eqref{i2n13} with  $(\AZ{0})^g_{\{V_0,W_0\}}$, it is deduced that
\begin{equation} \label{i2n17}
	\NZ{0}(\id_{\mathcal B}-\PP_{V_0}) = \left(\id_{\mathcal B} - \NZ{-2}\AZ{2} - \NZ{-1} \AZ{1}\right) (\AZ{0})^g_{\{V_0,W_0\}}. 
\end{equation}
From the definition of \(\KKK\), we have $\AZ{1}{\KKK} \subset \RR$. Therefore, it is easily deduced that 
\begin{align}\label{i2n18}
	\NZ{0} \AZ{1}{\mid_{\KKK}} = \NZ{0}(\id_{\mathcal B}-\PP_{V_0})  \AZ{1}{\mid_{\KKK}}.
\end{align} 
Combining \eqref{i2n16}, \eqref{i2n17} and \eqref{i2n18}, we have
\begin{equation}\label{i2n19}
	\left(\NZ{-2}  \AZ{3} + \NZ{-1} \AZ{2} + (\id_{\mathcal B} - \NZ{-2} \AZ{2} - \NZ{-1}  \AZ{1})(\AZ{0})^g_{\{V_0,W_0\}} \AZ{1} \right){\mid_{\KKK}} = 0.
\end{equation}
Rearranging terms, \eqref{i2n19} reduces to  
\begin{align} \label{i222eqq1}		\NZ{-1}A^\dagger_{2\,\{V_0,W_0\}}{\mid_{\KKK}} = &-\NZ{-2}  \left(\AZ{3}  - \AZ{2} (\AZ{0})^g_{\{V_0,W_0\}} \AZ{1} \right){\mid_{\KKK}}  - (\AZ{0})^g_{\{V_0,W_0\}} \AZ{1}{\mid_{\KKK}}.
\end{align}
Moreover with a trivial algebra, it can be shown that \eqref{i222eqq1} is equal to 
\begin{align} \label{i222eqq2}
	\NZ{-1}A^\dagger_{2\,\{V_0,W_0\}}{\mid_{\KKK}} = &-\NZ{-2}  \left( A^\dagger_{3\,\{V_0,W_0\}} - A^\dagger_{2\,\{V_0,W_0\}}(\AZ{0})^g_{\{V_0,W_0\}}\AZ{1} \right){\mid_{\KKK}} - (\AZ{0})^g_{\{V_0,W_0\}} \AZ{1}{\mid_{\KKK}}.
\end{align}	
From the identity decomposition \eqref{iddecomp}, we have $\NZ{-1} = \NZ{-1}(\id_{\mathcal B}-\PP_{V_0}) + \NZ{-1} (\id_{\mathcal B}-\PP_{V_1})\PP_{V_0}  + \NZ{-1} \PP_{V_1}$, so  \eqref{i222eqq2} can be written as follows: 
\begin{align}
	\NZ{-1} S^\dagger_{\{V_0,W_0,V_1\}}  =  &-\NZ{-2}  \left( A^\dagger_{3\,\{V_0,W_0\}} - A^\dagger_{2\,\{V_0,W_0\}}(\AZ{0})^g_{\{V_0,W_0\}}\AZ{1} \right){\mid_{\KKK}} - (\AZ{0})^g_{\{V_0,W_0\}} \AZ{1}{\mid_{\KKK}} \notag \\ 
	&-\NZ{-1}(\id_{\mathcal B}-\PP_{V_0})A^\dagger_{2\,\{V_0,W_0\}}{\mid_{\KKK}} - \NZ{-1} (\id_{\mathcal B}-\PP_{V_1})\PP_{V_0}A^\dagger_{2\,\{V_0,W_0\}}{\mid_{\mathcal K_{(z_0)}}}. \label{i222eqq3} 
\end{align}
We obtained explicit formulas for \(\NZ{-1}(\id_{\mathcal B}-\PP_{V_0})\) and \(\NZ{-1} (\id_{\mathcal B}-\PP_{V_1})\PP_{V_0}\)  in \eqref{i2n15} and \eqref{i2n151}. Moreover, we proved that  \( S^\dagger_{\{V_0,W_0,V_1\}} : \KKK \to \RRR\) is invertible. After some tedious algebra from \eqref{i222eqq3}, one can obtain the claimed formula for \(N_{-1}\) \eqref{n1claim} for our choice of complementary subspaces. Of course, the resulting \(N_{-1}\) needs to be understood as a map from $\mathcal B$ to $\mathcal B$.\\
\newline	
\textbf{Formulas for \((\NZ{j}, j \geq 0)\)} : Collecting the coefficients of $(z-1)^{j}$, $(z-1)^{j+1}$ and $(z-1)^{j+2}$ in the expansion of the identity \eqref{idexp2} when \(m=2\), we have
\begin{align}
	&G_j(0,2)+ \NZ{j} \AZ{0} = \indj, \label{ideneqq1}\\
	&G_j(1,2)+  \NZ{j} \AZ{1} + \NZ{j+1} \AZ{0}  = 0,  \label{ideneqq2} \\
	&G_j(2,2)+  \NZ{j} \AZ{2} + \NZ{j+1} \AZ{1} + \NZ{j+2}\AZ{0}  = 0.\label{ideneqq3}
\end{align}
From the identity decomposition \eqref{iddecomp}, similarly the operator \(\NZ{j}\) can be written as the sum of \(\NZ{j}(\id_{\mathcal B}-\PP_{V_0}) \), \(\NZ{j}(\id_{\mathcal B}-\PP_{V_1})\PP_{V_0}\), and \(\NZ{j}\PP_{V_1}\). We will find an explicit formula for each summand.
First from \eqref{ideneqq1}, it can be easily verified that 
\begin{align} \label{raneqq1}
	\NZ{j}(\id_{\mathcal B} -\PP_{V_0}) = \indj(\AZ{0})^g - G_j(0,2) (\AZ{0})^g_{\{V_0,W_0\}}.
\end{align}
By restricting the domain of \eqref{ideneqq2} to \(\KK\), we obtain
\begin{align} \label{raneqq2}
	\NZ{j} \AZ{1}{\mid_{\KK}} = - G_j(1,2){\mid_{\KK}}.
\end{align}
Using the identity decomposition \(\id_{\mathcal B} = \PP_{V_0} + (\id_{\mathcal B} -\PP_{V_0})\) and \eqref{raneqq1}, we may rewrite  \eqref{raneqq2} as follows.
\begin{align}
	\NZ{j} S_{V_0} = &- G_j(1,2){\mid_{\KK}} - \indj(\AZ{0})^g_{\{V_0,W_0\}} \AZ{1}{\mid_{\KK}} + G_j(0,2) (\AZ{0})^g_{\{V_0,W_0\}}\AZ{1}{\mid_{\KK}}. \label{raneqq3}
\end{align} 
Composing both sides of \eqref{raneqq3} with  \( \SSG\PP_{V_0}\), an explicit formula for \(\NZ{j}(\id_{\mathcal B} - \PP_{V_1})\PP_{V_0}  \) can be obtained as follows:
\begin{align}
	\NZ{j}(\id_{\mathcal B} - \PP_{V_1})\PP_{V_0} =& - G_j(1,2)\SSG\PP_{V_0} - \indj(\AZ{0})^g_{\{V_0,W_0\}} \AZ{1}\SSG\PP_{V_0} \nonumber \\ &  + G_j(0,2) (\AZ{0})^g_{\{V_0,W_0\}} \AZ{1}\SSG\PP_{V_0}. \label{raneqq6}
\end{align}
Restricting the domain of \eqref{ideneqq3} to $\KKK$, we obtain 
\begin{align}\label{raneqq4}
	G_j(2,2){\mid_{\KKK}}
	+  \NZ{j} \AZ{2}{\mid_{\KKK}} + \NZ{j+1} \AZ{1}{\mid_{\KKK}} = 0.
\end{align}
Composing  both sides of \eqref{ideneqq2} with  \((\AZ{0})^g_{\{V_0,W_0\}}\), it is easily deduced that
\begin{align}
	\NZ{j+1} (\id_{\mathcal B}-\PP_{V_0})  = &- G_j(1,2)(\AZ{0})^g_{\{V_0,W_0\}} -   \NZ{j} \AZ{1}(\AZ{0})^g_{\{V_0,W_0\}}. \label{raneqq5}
\end{align}
Note that we have \(\NZ{j+1} \AZ{1}{\mid_{\KKK}} = \NZ{j+1}(\id_{\mathcal B}-\PP_{V_0})\AZ{1}{\mid_{\KKK}} \) from the definition of \(\KKK\).  Combining this with \eqref{raneqq4} and \eqref{raneqq5}, we obtain the following equation.
\begin{align}
	\NZ{j}A^{\dagger}_{2\,\{V_0,W_0\}}{\mid_{\KKK}}  &= - G_j(2,2){\mid_{\KKK}}
	+  G_j(1,2)(\AZ{0})^g_{\{V_0,W_0\}}\AZ{1}{\mid_{\KKK}}. \label{raneqq8}
\end{align}
We know \(\NZ{j} = \NZ{j}(\id_{\mathcal B}-\PP_{V_0}) + \NZ{j}(\id_{\mathcal B}-\PP_{V_1})\PP_{V_0} + \NZ{j}\PP_{V_1}\), and  already obtained explicit formulas for the last two term. Substituting the obtained formulas into \eqref{raneqq8}, we obtain
\begin{align}
	\NZ{j}\PP_{V_1}A^{\dagger}_{2\,\{V_0,W_0\}}{\mid_{\KKK}} &=- G_j(2,2){\mid_{\KKK}} +  G_j(1,2)(\AZ{0})^g_{\{V_0,W_0\}}\AZ{1}{\mid_{\KKK}}   - \indj(\AZ{0})^g_{\{V_0,W_0\}}A^{\dagger}_{2\,\{V_0,W_0\}}{\mid_{\KKK}}  \notag  \\ 
	&\,\quad + G_j(0,2,) (\AZ{0})^g_{\{V_0,W_0\}}A^{\dagger}_{2\,\{V_0,W_0\}}{\mid_{\KKK}}  \nonumber   +  G_j(1,2)\SSG \PP_{V_0}A^{\dagger}_{2\,\{V_0,W_0\}}{\mid_{\KKK}}  \nonumber \\ 
	&\,\quad + \indj(\AZ{0})^g  \AZ{1}\SSG\PP_{V_0}A^{\dagger}_{2\,\{V_0,W_0\}}{\mid_{\KKK}} \nonumber \\&\,\quad- G_j(0,2)(\AZ{0})^g_{\{V_0,W_0\}} \AZ{1}\SSG \PP_{V_0} A^{\dagger}_{2\,\{V_0,W_0\}}{\mid_{\KKK}}.  \label{raneqq7}
\end{align}
Composing  both sides of \eqref{raneqq7} with \((S^\dagger_{\{V_0, W_0, V_1\}})^{-1}\PP_{V_1}\), we obtain the formula for \(\NZ{j}\PP_{V_1}\). Combining this formula with \eqref{raneqq1} and \eqref{raneqq6}, one can verify the claimed formula \eqref{njclaim} for our choice of complementary subspaces after some algebra. 

Even though our recursive formula is obtained under a given choice of complementary subspaces \(V_0,W_0,V_1\) and \(W_1\), we know, due to the uniqueness of the Laurent series, that it does not depend on our choice of complementary subspaces.
\end{proof}

\begin{remark} \label{rem1}
Let us narrow down our discussion to $\mathcal H$, a complex separable Hilbert space.
In $\mathcal H$, there is a canonical notion of a complementary subspace, called the orthogonal complement, while we do not have such a notion in $\mathcal{B}$. We therefore may let $(\ran \AZ{0})^\perp$ (resp.\ $(\ker \AZ{0})^\perp$) be our choice of \(\RC\) (resp.\ $\KC$).  Then \(\mathrm{P}_{(\ran \AZ{0})^\perp}\) and  \(\mathrm{P}_{(\ker \AZ{0})^\perp}\) are orthogonal projections. Then our generalized inverse \((\AZ{0})^g_{\{(\ran \AZ{0})^\perp,(\ker \AZ{0})^\perp\}}\) has the following properties:
\begin{align*}
	&(\AZ{0})^g_{\{(\ran \AZ{0})^\perp,(\ker \AZ{0})^\perp\}}\AZ{0} = (\id_{\mathcal H}-\PP_{(\ran \AZ{0})^\perp}),\\ &\AZ{0}(\AZ{0})^g_{\{(\ran \AZ{0})^\perp,(\ker \AZ{0})^\perp\}} = \PP_{(\ker \AZ{0})^\perp}. 
\end{align*}  
That is, both of $(\AZ{0})^g_{\{(\ran \AZ{0})^\perp,(\ker \AZ{0})^\perp\}}\AZ{0}$ and $\AZ{0}(\AZ{0})^g_{\{(\ran \AZ{0})^\perp,(\ker \AZ{0})^\perp\}}$  are self-adjoint operators, meaning that  \((\AZ{0})^g_{\{(\ran \AZ{0})^\perp,(\ker \AZ{0})^\perp\}}\) is the Moore-Penrose inverse operator of \(A_0\) \citep[Section 1]{Engl1981}. 
Moreover, we may let $	(\ran \AZ{0})^\perp \cap (S_{(\ran \AZ{0})^\perp}\KK)^\perp$ be our choice of \(\RCC\). 
This choice trivially satisfies \eqref{choice}, and it allows the orthogonal decomposition of \(\mathcal H\) as follows:
\begin{align*}
	\mathcal H = \RR \oplus_\perp S_{(\ran \AZ{0})^\perp}\KK \oplus_\perp \RCC. 
\end{align*}
Letting \(\KKK^\perp \cap \KK\) be our choice of  \(\KCC\), we can also make a generalized inverse of \(S_{(\ran \AZ{0})^\perp}\) become the Moore-Penrose inverse operator. This specific choice of complementary subspaces appears to be standard in $\mathcal H$ among many other possible choices.  \end{remark} 

\begin{remark}
	Under the specific choice of complementary subspaces in Remark \ref{rem1}, \cite{BS2018} stated and proved similar theorems to our Propositions \ref{propn1} and \ref{propn2}, without providing a recursive formula for $\NZ{j}$. The reader is referred to Theorem 3.1 and 3.2 of their paper for more details. On the other hand, we explicitly take all other possible choices of complementary subspaces into account and provide a recursive formula to obtain a closed-form expression of the Laurent series.  
	Therefore even if we restrict our concern to a Hilbert space setting, our propositions can be viewed as extended versions of those in \cite{BS2018}.
\end{remark}


\section{Representation theory} \label{srep}
In this section, we derive a suitable extension of the Granger-Johansen representation theory, which will be given as an application of the results established in Section \ref{sholomorphic}. Even if there are a few versions of this theorem developed in a possibly infinite dimensional Hilbert/Banach space (see e.g., \citealp{BSS2017,BS2018,Franchi2017b,albrecht2021resolution,seo_2022}), ours seems to be the first that can provide a full charaterization of I(1) and I(2) solutions (except a term depending on initial values) of a possibly infinite order autoregressive law of motion in a Banach space.

Let $A:\mathbb{C} \to \mathcal L_{\mathcal B}$ be a holomorphic operator pencil, then it allows the following Taylor series.
\begin{align*}
A(z) = \sum_{j=0}^\infty \AZ{j,(0)}z^j,
\end{align*}
where \( A_{j,(0)}\) denotes the coefficient of $z^j$ in the Taylor series of \(A(z)\) around \(0\). Note that we use additional subscript \((0)\) to distinguish it from \(A_j\) which denotes the coefficient of \((z-1)^j\) in the Taylor series of \(A(z)\) around \(1\). As in the previous sections,  we let \(N(z)\) denote  \(A(z)^{-1}\) if it exists.

Let $D_r \subset \mathbb{C}$ denote the open disk centered at the origin with radius $r > 0$ and $\overline{D}_{r}$ be its closure. Throughout this section, we employ the following assumption:
\begin{assumption}\label{assu2} \hspace{0.1cm}
\begin{itemize}
	\item[(i)]  $A:\mathbb{C} \to \mathcal L_{\mathcal B}$ is a holomorphic Fredholm pencil.
	\item[(ii)] $A(z)$ is invertible on $\overline{D}_{1}\setminus\{1\}$. 
\end{itemize} 
\end{assumption}

Now we provide one of the main results of this section.	 To simplify expressions in the following propositions, we keep using the notations introduced in Section \ref{sholomorphic}. Moreover, we introduce $\pi_j(k)$ for \(j \geq 0\), which is given by 
\begin{align*}
\pi_0(k) = 1, \quad \pi_1(k) = k, \quad \pi_j(k) = k(k-1)\cdots(k-j+1), \quad j \geq 2.
\end{align*}

\begin{proposition} \label{grt}
Suppose that \(A(z)\) satisfies Assumption \ref{assu2} and we have a sequence \((X_t, t \geq -p+1)\) satisfying 
\begin{align}
	\sum_{j=0}^\infty \AZ{j,(0)} X_{t-j} = \varepsilon_t, \label{oarlaw}
\end{align} where \(\varepsilon = (\varepsilon_t, t \in \mathbb{Z})\) is a strong white noise. Then the following conditions are equivalent to each other. 
\begin{itemize}
	\item[$\mathrm{(i)}$]  $A(z)^{-1}$ has a simple pole at $z=1$. 
	\item[$\mathrm{(ii)}$]  $\mathcal B = \RR \oplus \AZ{1} \KK$. 
	\item[$\mathrm{(iii)}$] For any choice of $\RC$ ,  $\SZ{1} : \KK \to \RC$ is invertible. 
	\item[$\mathrm{(iv)}$] For some choice of $\RC$ ,  $\SZ{1} : \KK \to \RC$ is invertible. 
\end{itemize}
Under any of these equivalent conditions, $X_t$ allows the representation: for some \(\tau_0\) depending on initial values, 
\begin{equation} \label{grteq}
	X_t = \tau_0 -\NO{-1} \sum_{s=1}^t \varepsilon_s  + \nu_t, \quad t \geq 0.
\end{equation} 
Moreover, \(\nu_t \in L^2_{\mathcal B}\) and satisfies 
\begin{align}\label{grteq2}
	&\nu_t = \sum_{j=0}^\infty \Phi_j \varepsilon_{t-j}, \quad \Phi_j = \sum_{k=j}^\infty(-1)^{k-j} \pi_j(k) \NO{k},	
\end{align}
where $(\NO{j}, j \geq -1)$ can be explicitly obtained from Proposition \ref{propn1}. 	
\end{proposition}
\begin{proof}  
\noindent Under Assumption \ref{assu2}, there exists \(\eta>0\) such that  \(A(z)^{-1}\) depends holomorphically on \(z \in D_{1+\eta} \setminus\{1\}\). To see this, note that the analytic Fredholm theorem implies that \(\sigma(A)\) is a discrete set. Since \(\sigma(A)\) is closed, it is deduced that \(\sigma(A) \cap \overline{D}_{1+r}\) is a closed discrete subset of \(\overline{D}_{1+r}\)  for some \(0<r<\infty\). The fact that $\overline{D}_{1+r}$ is a compact subset of \(\mathbb{C}\) implies that there are only finitely many elements in \(\sigma(A) \cap \overline{D}_{1+r}\). Furthermore since \(1\) is an isolated element of \(\sigma(A)\),  it can be easily deduced that there exists \(\eta \in (0,r)\) such that \(A(z)^{-1}\) depends holomorphically on \(z \in D_{1+\eta} \setminus\{1\}\).
Since \(1 \in \sigma(A)\) is an isolated element, the equivalence of conditions (i)-(iv) is implied by Proposition \ref{propn1}. 

Under any of the equivalent conditions, it is deduced from Proposition \ref{propn1} that  $N(z) = \NO{-1}(z-1)^{-1} + N^H(z)$, where $N^H(z)$ denotes the holomorphic part of the Laurent series. Moreover, we can explicitly obtain the coefficients \(( \NO{j}, j \geq -1)\) using the  recursive formula provided in Proposition \ref{propn1}. It is clear that \((1-z)N(z)\) can be holomorphically extended over $1$, and we can rewrite it as
\begin{align}\label{propeq1}
	(1-z) N(z)^{-1} = -\NO{-1}  + (1-z) N^H(z).
\end{align}
Applying the linear filter induced by \eqref{propeq1} to both sides of \eqref{oarlaw}, we obtain
\begin{align*} 
	\Delta X_t := X_t - X_{t-1} = -\NO{-1} \varepsilon_t + (\nu_t - \nu_{t-1}),
\end{align*}
where $\nu_s = \sum_{j=0}^\infty N_{j,(0)}^H   \varepsilon_{s-j}$, and \(N_{j,(0)}^H\) denotes the coefficient of $z^j$ in the Taylor series of \(N^H(z)\) around \(0\). Clearly the process
\begin{align*}
	X_t^* = -\NO{-1} \sum_{s=1}^t \varepsilon_s  + \nu_t
\end{align*}
is a solution, and the complete solution is obtained by adding the solution to $\Delta^2 X_t = 0$, which is given by \(\tau_0\). We then show $\nu_s$ is convergent in $L^2_H$. Note that    
\begin{align} 
	\left \| \sum_{j=0}^\infty N_{j,(0)}^H \varepsilon_{s-j} \right \|  \leq  \sum_{j=0}^\infty  \|N_{j,(0)}^H\|_{\mathcal L_{\mathcal B}} \|\varepsilon_{s-j}\| \leq C  \sum_{j=0}^\infty  \|N_{j,(0)}^H\|_{\mathcal L_{\mathcal B}}, 
	\label{normsum}\end{align} 
where \(C\) is some positive constant. The fact that $N^H(z)$ is holomorphic on $D_{1+\eta}$ implies that $\|N_{j,(0)}^H\|$ exponentially decreases as $j$ goes to infinity. This shows that the right-hand side of \eqref{normsum} converges to a finite quantity, so \(\nu_s\) converges in \(L^2_H\).

It is easy to verify \eqref{grteq2} from an elementary calculus. 
\end{proof}

\begin{remark}
Given that \(\varepsilon_t\) is a strong white noise, the sequence \((\nu_t, t \in \mathbb{Z})\) in our representation \eqref{grteq} is a stationary sequence. Therefore,  \eqref{grteq} shows that \(X_t\) can be decomposed into three different components; a random walk, a stationary process and a term that depends on initial values. 
\end{remark}

\begin{proposition} \label{grti2}
Suppose that \(A(z)\) satisfies Assumption \ref{assu2} and we have a sequence \((X_t, t \geq -p+1)\) satisfying \eqref{oarlaw}. Then the following conditions are equivalent to each other. 
\begin{itemize}
	\item[$\mathrm{(i)}$] $A(z)^{-1}$ has a second order pole at \(z=1\).  
	\item[$\mathrm{(ii)}$]  For some choice of $\RC$, $\KC$, we have 
	\begin{align*}
		\mathcal B = \RRR \oplus A^{\dagger}_{2\,\{\RC, \KC\}} \KKK.
	\end{align*}
	\item[$\mathrm{(iii)}$]   For any choice of $\RC$, $\KC$, and $\RCC$ satisfying \eqref{choice},
	$S^\dag_{\{\RC, \KC, \RCC\}} : \KKK \to \RCC \text{ is invertible. } $
	\item[$\mathrm{(iv)}$]   For some choice of $\RC$, $\KC$, and $\RCC$ satisfying \eqref{choice},
	$S^\dag_{\{\RC, \KC, \RCC\}} : \KKK \to \RCC \text{ is invertible. } $
\end{itemize}
Under any of these equivalent conditions, $X_t$ allows the representation: for some \(\tau_0\) and \(\tau_1\) depending on initial values,
\begin{align} 
	X_t = & \tau_0 + \tau_1 t  + \NO{-2}\sum_{\tau=1}^{t}\sum_{s=1}^\tau \varepsilon_s - \NO{-1}\sum_{s=1}^t \varepsilon_t + \nu_t, \quad t \geq 0. \label{grti2eq}
\end{align}
Moreover, \(\nu_t \in L^2_{\mathcal B}\) and satisfies 
\begin{align}\label{grteq22}
	&\nu_t = \sum_{j=0}^\infty \Phi_j \varepsilon_{t-j}, \quad \Phi_j = \sum_{k=j}^\infty(-1)^{k-j} \pi_j(k) \NO{k},	
\end{align}
where $(\NO{j}, j \geq -2)$ can be explicitly obtained from Proposition \ref{propn2}. 	
\begin{proof}
	As we showed in Proposition \ref{grt}, we know there exists \(\eta>0\) such that  \(A(z)^{-1}\) depends holomorphically on \(z \in D_{1+\eta} \setminus\{1\}\). Due to Proposition \ref{propn2}, we know \(N(z) = \NO{-2}(z-1)^{-2} + \NO{-1}(z-1)^{-1} + N^H(z)\), where \(N^H(z)\) is the holomorphic part of the Laurent series.  
	
	\((1-z)^2A(z)^{-1}\) can be holomorphically extended over $1$ so that it  holomorphic on \(D_{1+\eta}\).  Then we have 
	\begin{align*}
		(1-z)^2 N(z)^{-1} = \NZ{-2} - \NZ{-1}(1-z) + (1-z)^2 N^{H}(z).
	\end{align*}
	Applying the linear filter induced by \((1-z)^2 A(z)^{-1}\) to both sides of \eqref{oarlaw}, we obtain for \(s=1,\ldots, t\)
	\begin{align*} 
		\Delta^2 X_t = \NO{-2} \varepsilon_t - \NO{-1} \Delta \varepsilon_t  + (\Delta \nu_t - \Delta \nu_{t-1}),
	\end{align*}
	where  $\nu_t := \sum_{j} \NC{j,(0)}^H \varepsilon_{t-j}$. From \eqref{normsum}, we know \(\nu_t\) converges in $L^2_{\mathcal B}$. Clearly the process 
	\begin{align*} 
		X_t^* =	\NO{-2}\sum_{\tau=1}^{t}\sum_{s=1}^\tau \varepsilon_s - \NO{-1}\sum_{s=1}^t \varepsilon_t + \nu_t 
	\end{align*}
	is a solution. Since the solution to \(\Delta^2 X_t = 0\) is given by \(\tau_0 + \tau_1 t\), we obtain \eqref{grti2eq}. It is also easy to verify \eqref{grteq22} from an elementary calculus.
\end{proof}
\end{proposition}

\begin{remark}
Similarly, the sequence \((\nu_t, t \in \mathbb{Z})\) in our representation \eqref{grti2eq} is stationary given that $\varepsilon$ is a strong white noise. Then the representation \eqref{grti2eq} shows that \(X_t\) can be decomposed into a cumulative random walk, a random walk, a stationary process and a term that depends on initial values. 
\end{remark}

\begin{remark}
	From the analytic Fredholm theorem, we know that the random walk component in our I(1) and I(2) representation takes values in a finite dimensional space, which is similar to the existing results  by \cite{BS2018} and \cite{Franchi2017b}; for statistical inference on function-valued time series containing a random walk component, the component is often assumed to be finite dimensional and the representation results presented by \cite{BS2018} and \cite{Franchi2017b} are used to justify this assumption (see e.g., \citealp{seo2020functional} and \citealp{NSS}). 
\end{remark}

\begin{remark}
Propositions \ref{grt} and \ref{grti2} require the autoregressive law of motion to be characterized by a holomorphic operator pencil satisfying Assumption \ref{assu2}. However, we expect a wide class of autoregressive processes considered in practice satisfies the requirement. For example, for \(p \in \mathbb{N}\), let \(\Phi_1, \ldots, \Phi_p\) be compact operators. Then the autoregressive law of motion given by 
\begin{align*}
	X_t = \sum_{j=1}^p \Phi_j X_{t-j} + \varepsilon_t
\end{align*}
satisfies the requirement.
\end{remark}

\begin{remark}
Even though we have assumed that $\varepsilon$ is a strong white noise for simplicity, we may allow more general innovations in Proposition \ref{grt} and \ref{grti2}. For example, we could allow $\|\varepsilon_t\|$ to depend on $t$. Even in this case, if $\|\varepsilon_t\|$ is bounded by $a + |t|^b $ for some $a,b \in \mathbb{R}$, the right hand side of \eqref{normsum} is still bounded by a finite quantity, meaning that \(\nu_t\) converges in \(L^2_H\).    
\end{remark}

\begin{remark}
For simplicity, we have only considered purely stochastic process in Proposition \ref{grt}. However, the inclusion of a deterministic component does not cause significant difficulties. For example, suppose that we have $(X_t, t \geq -p+1)$ generated by the following autoregressive law of motion.
\begin{align*}
	\sum_{j=0}^\infty A_{j,(0)} X_{t-j} = \gamma_t +  \varepsilon_{t}, \quad t \geq 1,
\end{align*}
where \((\gamma_t, t \in \mathbb{Z})\) is a deterministic sequence. In this case, we may need some condition on $(\gamma_t, t \in \mathbb{Z})$ for $\nu_t$ to converges in $L^2_{\mathcal B}$. We could assume that $\|\gamma_t\|$ is bounded by $a + |t|^b $ for some $a,b \in \mathbb{R}$.
\end{remark}

\section{Conclusion}\label{sconclude}
This paper considers inversion of a holomorphic Fredholm pencil based on the analytic Fredholm theorem. We obtain necessary and sufficient conditions for the inverse of a Fredholm operator pencil to have a simple pole and a second order pole, and further derive a closed-form expression of the Laurent expansion of the inverse around an isolated singularity. Using the results, we obtain a suitable version of the Granger-Johansen representation theorem in a general Banach space setting, which fully characterizes I(1) (and I(2)) solutions except a term depending on initial values.

\bibliographystyle{apalike}

\end{document}